\documentclass[reqno]{amsart}

\usepackage{amssymb}
\usepackage{hyperref}
\usepackage{mathtools}
\usepackage[all]{xy}
\usepackage{ifthen}
\usepackage{xargs}
\usepackage{bussproofs}
\usepackage{turnstile}

\hypersetup{colorlinks=true,linkcolor=blue}

\renewcommand{\turnstile}[6][s]
    {\ifthenelse{\equal{#1}{d}}
        {\sbox{\first}{$\displaystyle{#4}$}
        \sbox{\second}{$\displaystyle{#5}$}}{}
    \ifthenelse{\equal{#1}{t}}
        {\sbox{\first}{$\textstyle{#4}$}
        \sbox{\second}{$\textstyle{#5}$}}{}
    \ifthenelse{\equal{#1}{s}}
        {\sbox{\first}{$\scriptstyle{#4}$}
        \sbox{\second}{$\scriptstyle{#5}$}}{}
    \ifthenelse{\equal{#1}{ss}}
        {\sbox{\first}{$\scriptscriptstyle{#4}$}
        \sbox{\second}{$\scriptscriptstyle{#5}$}}{}
    \setlength{\dashthickness}{0.111ex}
    \setlength{\ddashthickness}{0.35ex}
    \setlength{\leasturnstilewidth}{2em}
    \setlength{\extrawidth}{0.2em}
    \ifthenelse{%
      \equal{#3}{n}}{\setlength{\tinyverdistance}{0ex}}{}
    \ifthenelse{%
      \equal{#3}{s}}{\setlength{\tinyverdistance}{0.5\dashthickness}}{}
    \ifthenelse{%
      \equal{#3}{d}}{\setlength{\tinyverdistance}{0.5\ddashthickness}
        \addtolength{\tinyverdistance}{\dashthickness}}{}
    \ifthenelse{%
      \equal{#3}{t}}{\setlength{\tinyverdistance}{1.5\dashthickness}
        \addtolength{\tinyverdistance}{\ddashthickness}}{}
        \setlength{\verdistance}{0.4ex}
        \settoheight{\lengthvar}{\usebox{\first}}
        \setlength{\raisedown}{-\lengthvar}
        \addtolength{\raisedown}{-\tinyverdistance}
        \addtolength{\raisedown}{-\verdistance}
        \settodepth{\raiseup}{\usebox{\second}}
        \addtolength{\raiseup}{\tinyverdistance}
        \addtolength{\raiseup}{\verdistance}
        \setlength{\lift}{0.8ex}
        \settowidth{\firstwidth}{\usebox{\first}}
        \settowidth{\secondwidth}{\usebox{\second}}
        \ifthenelse{\lengthtest{\firstwidth = 0ex}
            \and
            \lengthtest{\secondwidth = 0ex}}
                {\setlength{\turnstilewidth}{\leasturnstilewidth}}
                {\setlength{\turnstilewidth}{2\extrawidth}
        \ifthenelse{\lengthtest{\firstwidth < \secondwidth}}
            {\addtolength{\turnstilewidth}{\secondwidth}}
            {\addtolength{\turnstilewidth}{\firstwidth}}}
        \ifthenelse{\lengthtest{\turnstilewidth < \leasturnstilewidth}}{\setlength{\turnstilewidth}{\leasturnstilewidth}}{}
    \setlength{\turnstileheight}{1.5ex}
    \sbox{\turnstilebox}
    {\raisebox{\lift}{\ensuremath{
        \makever{#2}{\dashthickness}{\turnstileheight}{\ddashthickness}
        \makehor{#3}{\dashthickness}{\turnstilewidth}{\ddashthickness}
        \hspace{-\turnstilewidth}
        \raisebox{\raisedown}
        {\makebox[\turnstilewidth]{\usebox{\first}}}
            \hspace{-\turnstilewidth}
            \raisebox{\raiseup}
            {\makebox[\turnstilewidth]{\usebox{\second}}}
        \makever{#6}{\dashthickness}{\turnstileheight}{\ddashthickness}}}}
        \mathrel{\usebox{\turnstilebox}}}

\newcommand{\newref}[4][]{
\ifthenelse{\equal{#1}{}}{\newtheorem{h#2}[hthm]{#4}}{\newtheorem{h#2}{#4}[#1]}
\expandafter\newcommand\csname r#2\endcsname[1]{#3~\ref{#2:##1}}
\expandafter\newcommand\csname R#2\endcsname[1]{#4~\ref{#2:##1}}
\expandafter\newcommand\csname n#2\endcsname[1]{\ref{#2:##1}}
\newenvironmentx{#2}[2][1=,2=]{
\ifthenelse{\equal{##2}{}}{\begin{h#2}}{\begin{h#2}[##2]}
\ifthenelse{\equal{##1}{}}{}{\label{#2:##1}}
}{\end{h#2}}
}

\newref[section]{thm}{theorem}{Theorem}
\newref{lem}{lemma}{Lemma}
\newref{prop}{proposition}{Proposition}
\newref{cor}{corollary}{Corollary}
\newref{cond}{condition}{Condition}
\newref{conj}{conjecture}{Conjecture}

\theoremstyle{definition}
\newref{defn}{definition}{Definition}
\newref{example}{example}{Example}

\theoremstyle{remark}
\newref{rem}{remark}{Remark}

\newcommand{\deq}{\equiv}
\newcommand{\repl}{:=}
\newcommand{\idtype}{\rightsquigarrow}

\newcommand{\cat}[1]{\mathbf{#1}}
\newcommand{\C}{\cat{C}}
\newcommand{\Mod}[1]{#1\text{-}\cat{Mod}}
\newcommand{\Th}{\cat{Th}}
\newcommand{\emptyCtx}{\mathbf{1}}

\newcommand{\we}{\mathcal{W}}
\newcommand{\I}{\mathrm{I}}
\newcommand{\J}{\mathrm{J}}
\newcommand{\class}[2]{#1\text{-}\mathrm{#2}}
\newcommand{\Icell}[1][\I]{\class{#1}{cell}}
\newcommand{\Icof}[1][\I]{\class{#1}{cof}}
\newcommand{\Jcell}[1][]{\Icell[\J#1]}
\newcommand{\cyli}{i}

\numberwithin{figure}{section}

\begin{document}

\title{Model Structures on Categories of Models of Type Theories}

\author{Valery Isaev}

\begin{abstract}
Models of dependent type theories are contextual categories with some additional structure.
We prove that if a theory $T$ has enough structure, then the category $\Mod{T}$ of its models carries the structure of a model category.
We also show that if $T$ has $\Sigma$ types, then weak equivalences can be characterized in terms of homotopy categories of models.
\end{abstract}

\maketitle

\section{Introduction}

It is well-known that algebraic models (such as categories with attributes \cite{pitts}, categories with families \cite{cwf} and contextual categories \cite{GAT})
of dependent type theories are related to categories with additional structure.
For example, it was prove in \cite{ext-eq} that models of the type theory with extensional $Id$ and $\Sigma$ types are equivalent (in a weak bicategorical sense)
to finitely complete categories, and if we assume $\Pi$ types, then we obtain an equivalence with locally cartesian closed categories.

Ideas of homotopy type theory suggest that models of dependent type theories with intensional $Id$ types should be related to $\infty$-categories.
There are several results (for example, \cite{shul-inv}, \cite{local-universes}, \cite{kapulkin}) that support this intuition,
but to make this relationship precise we need an appropriate definition of equivalences of models of type theories.

The main contribution of this paper is the construction of a model structure on categories of models of dependent type theories.
We define this model structure for every algebraic dependent type theory (as defined in \cite{alg-tt}) which has enough structure
(essentially, path types and a weak form of the univalence axiom).
Let $T_\Sigma$ be the theory with path types and $\Sigma$ types (see subsection~\ref{sec:sigma} for a precise definition).
Then we can state a formal conjecture:
\begin{conj}
The $(\infty,1)$-category presented by model category $\Mod{T_\Sigma}$ is equivalent to the $(\infty,1)$-category of finitely complete $(\infty,1)$-categories.
\end{conj}
Analogous conjectures can be stated for other theories such as the theory with path types, $\Sigma$ types and $\Pi$ types.

It was shown in \cite{szumilo} that the fibration category of fibration categories is equivalent to the fibration category of finitely complete quasicategories.
Thus it is natural to study the relationship between fibration categories and models of $T_\Sigma$.
For every model $X$ of $T_\Sigma$, we can define a fibration category $U(X)$ (see \cite{tt-fibr-cat}), and this correspondence defines a functor $U : \Mod{T_\Sigma} \to \cat{FibCat}$
from the category of models of $T_\Sigma$ to the category of fibration categories.
Since both $\Mod{T_\Sigma}$ and $\cat{FibCat}$ are fibration categories, it is natural to conjecture that $U$ is exact, but it seems that it is not.
The main problem is that it seems that $U$ does not preserve fibrations.
Nevertheless, we can show that $U$ preserves and reflects weak equivalences (see \rprop{sigma-we-ho}), which indicates that $\Mod{T}$ has the correct class of weak equivalences.

We will describe a theory with the interval type and define a model structure on the category of models of theories that have the interval type.
We can define usual $Id$ types in this theory, but it is stronger than the theory of $Id$ types.
For example, function extensionality holds in this theory.
It might be possible to define a model structure on the category of models of theories with $Id$ types,
but it is more difficult, and we do not know how to do it.

The paper is organized as follows.
In section~\ref{sec:HoTT-I}, we define different theories with the interval type and describe several constructions in these theories.
We also define a weak univalence axiom and prove that it implies a part of the usual version of this axiom.
In section~\ref{sec:model-structure}, we define a model structure on the category of models of a dependent type theory
and prove a characterization of weak equivalences in this category.

\section{Theories with an interval type}
\label{sec:HoTT-I}

In this section we describe the theory of an interval type.
We describe several constructions in this theory which we will need later.
In particular, we will show that theories with an interval type and path types also have $Id$ types.
We will use a (slightly informal) named presentation of terms,
from which a formal presentation in terms of De Bruijn indices can be recovered.

We will write $T_1 + T_2$ for the union of theories $T_1$ and $T_2$.
That is $T_1 + T_2 = T_1 \amalg_{T} T_2$, where $T$ is the common subtheory of $T_1$ and $T_2$.
Sometimes we will write $T_1 + T_2$ even if $T_1$ is a subtheory of $T_2$ (in this case, $T_1 + T_2 = T_2$).
For example, it is convenient to use this notation when $T_2$ is $T_1$ together with some additional axiom.

Theories with an interval type are closely related to theories with identity types.
So, let us first recall its definition from \cite{alg-tt}.
Theory $Id$ is a regular theory with the following function symbols:
\begin{align*}
Id & : (tm,n) \times (tm,n) \to (ty,n) \\
refl & : (tm,n) \to (tm,n) \\
J & : (ty,n+3) \times (tm,n+1) \times (tm,n) \times (tm,n) \times (tm,n) \to (tm,n)
\end{align*}
and the following axioms:
\medskip
\begin{center}
\AxiomC{$\Gamma \vdash ty(a) \deq ty(a')$}
\UnaryInfC{$\Gamma \vdash Id(a, a')\ type$}
\DisplayProof
\quad
\AxiomC{}
\UnaryInfC{$\Gamma \vdash refl(a) : Id(a, a)$}
\DisplayProof
\end{center}

\medskip
\begin{center}
\AxiomC{$\Gamma, x : A, y : A, z : Id(x,y) \vdash D\ type$}
\AxiomC{$\Gamma, x : A \vdash d : D'$}
\AxiomC{$\Gamma \vdash p : Id(a,a')$}
\TrinaryInfC{$\Gamma \vdash J(D,d,a,a',p) : D[a,a',p]$}
\DisplayProof
\end{center}
where $D' = D[y \repl x, z \repl refl(x)]$ and $A = ty(a)$.

\medskip
\begin{center}
\AxiomC{$\Gamma, x : ty(a), y : ty(a), z : Id(x,y) \vdash D\ type$}
\AxiomC{$\Gamma, x : ty(a) \vdash d : D'$}
\BinaryInfC{$\Gamma \vdash J(D,d,a,a,refl(a)) \deq d[a]$}
\DisplayProof
\end{center}
\medskip
where $D' = D[y \repl x, z \repl refl(x)]$.

We will also need slightly weaker version of $Id$ which we will denote by $Id_-$.
It has all of the function symbols of $Id$ together with the following one:
\[ Jeq : (tm,n+3) \times (tm,n+1) \times (tm,n) \to (tm,n) \]
Theory $Id_-$ has all of the axioms of $Id$ except the last one; instead it has the following additional axiom:
\begin{center}
\AxiomC{$\Gamma, x : A, y : A, z : Id(x,y) \vdash D\ type$}
\AxiomC{$\Gamma, x : A \vdash d : D'$}
\AxiomC{$\Gamma \vdash a : A$}
\TrinaryInfC{$\Gamma \vdash Jeq(D,d,a) : Id(J(D,d,a,a,refl(a)),d[a])$}
\DisplayProof
\end{center}
where $D' = D[y \repl x, z \repl refl(x)]$.
The idea is that the last axiom of $Id$ holds in $Id_-$ only propositionally.

Now, we can define the theory of the interval type.
Actually, there are several different ways to define such theory.
These theories are not isomorphic, but should be equivalent in some weaker sense.
First, let us define the most basic theory which has only the interval type and its constructors, but lacks any kind of eliminator for it.
Theory $I$ is a regular theory with function symbols $I : (ty,n)$, $left : (tm,n)$, $right : (tm,n)$, and the following axioms:
\begin{center}
\AxiomC{}
\UnaryInfC{$\Gamma \vdash I\ type$}
\DisplayProof
\quad
\AxiomC{}
\UnaryInfC{$\Gamma \vdash left : I$}
\DisplayProof
\quad
\AxiomC{}
\UnaryInfC{$\Gamma \vdash right : I$}
\DisplayProof
\end{center}

There are at least three different ways to define an eliminator for $I$.
The idea always the same: given a fibration over $I$ and a point in the fibre over
some point $i : I$, we can transport it to the fibre over some other point $j : I$.
In different eliminators, we can take different $i$ and $j$.
In $coe_0$, we can only take $i = left$ and $j = right$.

\medskip
\begin{center}
\AxiomC{$\Gamma, x : I \vdash D\ type$}
\AxiomC{$\Gamma \vdash d : D[x \repl left]$}
\BinaryInfC{$\Gamma \vdash coe_0(\lambda x.\,D, d) : D[x \repl right]$}
\DisplayProof
\end{center}

In $coe_1$, we can take $i = left$ and arbitrary $j$.

\medskip
\begin{center}
\AxiomC{$\Gamma, x : I \vdash D\ type$}
\AxiomC{$\Gamma \vdash d : D[x \repl left]$}
\AxiomC{$\Gamma \vdash i : I$}
\TrinaryInfC{$\Gamma \vdash coe_1(\lambda x.\,D, d, i) : D[x \repl i]$}
\DisplayProof
\end{center}

\medskip
\begin{center}
\AxiomC{$\Gamma, x : I \vdash D\ type$}
\AxiomC{$\Gamma \vdash d : D[x \repl left]$}
\BinaryInfC{$\Gamma \vdash coe_1(\lambda x.\,D, d, left) \deq d$}
\DisplayProof
\end{center}

In $coe_2$, both $i$ and $j$ may be arbitrary.

\medskip
\begin{center}
\AxiomC{$\Gamma, x : I \vdash D\ type$}
\AxiomC{$\Gamma \vdash i : I$}
\AxiomC{$\Gamma \vdash d : D[x \repl i]$}
\AxiomC{$\Gamma \vdash j : I$}
\QuaternaryInfC{$\Gamma \vdash coe_2(\lambda x.\,D, i, d, j) : D[x \repl j]$}
\DisplayProof
\end{center}

\medskip
\begin{center}
\AxiomC{$\Gamma, x : I \vdash D\ type$}
\AxiomC{$\Gamma \vdash d : D[x \repl left]$}
\BinaryInfC{$\Gamma \vdash coe_2(\lambda x.\,D, left, d, left) \deq d$}
\DisplayProof
\end{center}

It turns out that $coe_0$ is too weak.
To make it equivalent to other two theories, we need to add regular theory $sq$ to it,
which has one function symbol $sq : (tm,n) \times (tm,n) \to (tm,n)$ and the following axioms:
\medskip
\begin{center}
\AxiomC{$\Gamma \vdash i : I$}
\AxiomC{$\Gamma \vdash j : I$}
\BinaryInfC{$\Gamma \vdash sq(i,j) : I$}
\DisplayProof
\qquad
\AxiomC{$\Gamma \vdash i : I$}
\UnaryInfC{$\Gamma \vdash sq(i,left) \deq left$}
\DisplayProof
\end{center}

\medskip
\begin{center}
\AxiomC{$\Gamma \vdash j : I$}
\UnaryInfC{$\Gamma \vdash sq(left,j) \deq left$}
\DisplayProof
\qquad
\AxiomC{$\Gamma \vdash j : I$}
\UnaryInfC{$\Gamma \vdash sq(right,j) \deq j$}
\DisplayProof
\end{center}

Also, these theories correspond to $Id_-$.
To get theories that correspond to $Id$, we need to add one additional rule to each of them:
\medskip
\begin{center}
\AxiomC{$\Gamma \vdash a : A$}
\UnaryInfC{$\Gamma \vdash coe_0(\lambda x.\,A, a) \deq a$}
\DisplayProof
\quad
\AxiomC{$\Gamma \vdash a : A$}
\UnaryInfC{$\Gamma \vdash coe_1(\lambda x.\,A, a, right) \deq a$}
\DisplayProof
\end{center}

\medskip
\begin{center}
\AxiomC{$\Gamma \vdash a : A$}
\UnaryInfC{$\Gamma \vdash coe_2(\lambda x.\,A, left, a, right) \deq a$}
\DisplayProof
\end{center}
\smallskip
We denote these theories by $coe_0 + \sigma$, $coe_1 + \sigma$ and $coe_2 + \sigma$.

We will also consider additional axioms $\beta_1$ and $\beta_2$ for $coe_2$.
Axiom $\beta_1$ is defined as follows:
\medskip
\begin{center}
\AxiomC{$\Gamma, x : I \vdash D\ type$}
\AxiomC{$\Gamma \vdash d : D[x \repl left]$}
\BinaryInfC{$\Gamma \vdash coe_2(\lambda x.\,D, right, d, right) \deq d$}
\DisplayProof
\end{center}
\medskip
Axiom $\beta_2$ is defined as follows:
\medskip
\begin{center}
\AxiomC{$\Gamma, x : I \vdash D\ type$}
\AxiomC{$\Gamma \vdash d : D[x \repl left]$}
\AxiomC{$\Gamma \vdash i : I$}
\TrinaryInfC{$\Gamma \vdash coe_2(\lambda x.\,D, i, d, i) \deq d$}
\DisplayProof
\end{center}
Obviously, we have maps from $coe_2$ to $coe_2 + \beta_1$ and from $coe_2 + \beta_1$ to $coe_2 + \beta_2$.

Theory $coe_2 + \beta_2$ is slightly stronger than other theories.
To make them equivalent to $coe_2 + \beta_2$, we need to assume additional operations.
For example, we can consider regular theory $dc$ which has one function symbol
$dc : (tm,n) \times (tm,n) \times (tm,n) \to (tm,n)$ and the following axiom:
\medskip
\begin{center}
\AxiomC{$\Gamma \vdash i : I$}
\AxiomC{$\Gamma \vdash j : I$}
\AxiomC{$\Gamma \vdash k : I$}
\TrinaryInfC{$\Gamma \vdash dc(i,j,k) : I$}
\DisplayProof
\qquad
\AxiomC{$\Gamma \vdash i : I$}
\AxiomC{$\Gamma \vdash j : I$}
\BinaryInfC{$\Gamma \vdash dc(i,j,left) \deq i$}
\DisplayProof
\end{center}

\medskip
\begin{center}
\AxiomC{$\Gamma \vdash i : I$}
\AxiomC{$\Gamma \vdash j : I$}
\BinaryInfC{$\Gamma \vdash dc(i,j,right) \deq j$}
\DisplayProof
\qquad
\AxiomC{$\Gamma \vdash i : I$}
\AxiomC{$\Gamma \vdash k : I$}
\BinaryInfC{$\Gamma \vdash dc(i,i,k) \deq i$}
\DisplayProof
\end{center}

\subsection{Homogeneous path types}

To define maps between theories with $Id$ types and theories with an interval type,
we need to add an additional construction to the latter, which we call \emph{homogeneous path types}.
Let $HPath$ be a regular theory with function symbols $\idtype\ : (tm,n) \times (tm,n) \to (ty,n)$,
$path : (ty,n) \times (tm,n+1) \to (tm,n)$, and $at : (tm,n) \times (tm,n) \times (tm,n) \times (tm,n) \to (tm,n)$, and the following axioms:
\begin{center}
\AxiomC{$\Gamma \vdash a : A$}
\AxiomC{$\Gamma \vdash a' : A$}
\BinaryInfC{$\Gamma \vdash a \idtype a'\ type$}
\DisplayProof
\end{center}

\smallskip
\begin{center}
\AxiomC{$\Gamma \vdash A\ type$}
\AxiomC{$\Gamma, x : I \vdash a : A$}
\BinaryInfC{$\Gamma \vdash path(A, \lambda x.\,a) : a[x \repl left] \idtype a[x \repl right]$}
\DisplayProof
\end{center}

\smallskip
\begin{center}
\AxiomC{$\Gamma \vdash a : A$}
\AxiomC{$\Gamma \vdash p : a \idtype a'$}
\AxiomC{$\Gamma \vdash i : I$}
\TrinaryInfC{$\Gamma \vdash at(a, a', p, i) : A$}
\DisplayProof
\end{center}

\smallskip
\begin{center}
\AxiomC{$\Gamma \vdash A\ type$}
\AxiomC{$\Gamma, x : I \vdash a : A$}
\AxiomC{$\Gamma \vdash i : I$}
\TrinaryInfC{$\Gamma \vdash at(a[x \repl left], a[x \repl right], path(A, \lambda x.\,a), i) \deq a[x \repl i]$}
\DisplayProof
\end{center}

\smallskip
\begin{center}
\AxiomC{$\Gamma \vdash p : a \idtype a'$}
\UnaryInfC{$\Gamma \vdash path(ty(a), \lambda x.\,at(a, a', p, x)) \deq p$}
\DisplayProof
\end{center}

\smallskip
\begin{center}
\AxiomC{$\Gamma \vdash p : a \idtype a'$}
\UnaryInfC{$\Gamma \vdash at(a, a', p, left) \deq a$}
\DisplayProof
\quad
\AxiomC{$\Gamma \vdash p : a \idtype a'$}
\UnaryInfC{$\Gamma \vdash at(a, a', p, right) \deq a'$}
\DisplayProof
\end{center}

We can summarize the relationship between different theories in the following (noncommutative) diagram of theories:
\[ \xymatrix{ & Id_- \ar[rr] \ar[d] & & Id \ar[d] \\
              coe'_0 + sq \ar[r] & coe'_1 \ar[r] & coe'_2 \ar[d]    \ar[r] & coe'_0 + \sigma + sq \ar@<1.5pt>[r] \ar[d] & coe'_1 + \sigma \ar[d] \ar@<1.5pt>[l] \\
              sq          \ar[r] & dc     \ar[r] & coe'_2 + \beta_2 \ar[r] & coe'_0 + \sigma + dc \ar@<1.5pt>[r]        & coe'_1 + \sigma + dc   \ar@<1.5pt>[l]
            } \]
where $coe'_\alpha = coe_\alpha + HPath$.
This diagram does not commute strictly, but it should commute up to some appropriately defined notion of homotopy between morphisms of theories.

Arrows $Id_- \to Id$ and $coe'_1 \to coe'_2$ are obvious.
Let us construct vertical maps.
Define $refl$ and $J$ as follows:
\begin{align*}
refl(a) & = path(ty(a), \lambda x.\,a) \\
J(A, \lambda x y z.\,D, \lambda x.\,d, a, a', p) & = coe_0(\lambda i.\,D', d[x \repl a])
\end{align*}
where $D'$ is defined as follows:
\[ D[x \repl a, y \repl at(a, a', p, sq(i,right)), z \repl path(ty(a), \lambda j.\,at(a, a', p, sq(i,j)))]. \]
Note that $J(A, \lambda x y z.\,D, \lambda x.\,d, a, a, refl(a))$ equals to $coe_0(\lambda i. D'', d[x \repl a])$ where $D'' = D[x \repl a, y \repl a, z \repl refl(a)]$.
Thus if we have $\sigma$ rule, then equation $\vdash J(A, \lambda x y z.\,D, \lambda x.\,d, a, a, refl(a)) \deq d[x \repl a]$ holds.
If we have $coe_1$, then we can define $Jeq(D,d,a)$ as $coe_0(\lambda j.\,coe_1(\lambda i. D'', d[a], j) \idtype d[a], refl(d[a]))$.

Map $sq \to coe'_1$ can be defined as follows:
\[ sq(i,j) = at(left, j, coe_1(\lambda x.\,left \idtype x, refl(left), j), i). \]

Now, let us define arrow $coe_2 \to coe'_0 + sq + \sigma$.
First, note that we can define a map $coe_1 \to coe_0 + sq + \sigma$ as follows: $coe_1(\lambda x.\,D, d, j) = coe_0(\lambda i.\,D[x \repl sq(i,j)], d)$.
Then let $Ic(i) = path(I, \lambda j.\,sq(j,i))$.
Since $Ic(i) : left \idtype i$, we can define a term $dc'(i,j)$ of type $i \idtype j$.
Now, let $coe_2(\lambda x.\,D, i, d, j)$ be equal to $coe_0(\lambda x.\,D[x \repl at(i, j, dc'(i,j), x)], d)$.
Note that $\Gamma \vdash dc'(left,left) \deq refl(left)$; hence $\Gamma \vdash coe_2(\lambda x.\,D, left, d, left) \deq d$.

Maps in the bottom row are easy to define:
\begin{align*}
sq(i,j) & = dc(left,j,i) \\
dc(i,j,k) & = at(i,j,coe_2(\lambda x.\,i \idtype x, i, refl(i), j),k) \\
coe_2(\lambda x.\,D, i, d, j) & = coe_0(\lambda x.\,D[x \repl dc(i,j,x)], d)
\end{align*}

\subsection{Heterogeneous path types}

\emph{Heterogeneous path types} are a useful generalization of homogeneous path types.
Theory $Path$ be a theory with function symbols $Path : (ty,n+1) \times (tm,n) \times (tm,n) \to (ty,n)$,
    $path : (tm,n+1) \to (tm,n)$, and $at : (ty,n+1) \times (tm,n) \times (tm,n) \times (tm,n) \times (tm,n) \to (tm,n)$, and the following axioms:
\begin{center}
\AxiomC{$\Gamma, x : I \vdash A\ type$}
\AxiomC{$\Gamma \vdash a : A[x \repl left]$}
\AxiomC{$\Gamma \vdash a' : A[x \repl right]$}
\TrinaryInfC{$\Gamma \vdash Path(\lambda x.\, A, a, a')\ type$}
\DisplayProof
\end{center}

\smallskip
\begin{center}
\AxiomC{$\Gamma, x : I \vdash a : A$}
\UnaryInfC{$\Gamma \vdash path(\lambda x.\,a) : Path(\lambda x.\,A, a[x \repl left], a[x \repl right])$}
\DisplayProof
\end{center}

\smallskip
\begin{center}
\AxiomC{$\Gamma \vdash p : Path(\lambda x.\,A, a, a')$}
\AxiomC{$\Gamma \vdash i : I$}
\BinaryInfC{$\Gamma \vdash at(\lambda x.\,A, a, a', p, i) : A[x \repl i]$}
\DisplayProof
\end{center}

\smallskip
\begin{center}
\AxiomC{$\Gamma, x : I \vdash a : A$}
\AxiomC{$\Gamma \vdash i : I$}
\BinaryInfC{$\Gamma \vdash at(\lambda x.\,A, a[x \repl left], a[x \repl right], path(\lambda x.\,a), i) \deq a[x \repl i]$}
\DisplayProof
\end{center}

\smallskip
\begin{center}
\AxiomC{$\Gamma \vdash p : Path(\lambda x.\,A, a, a')$}
\UnaryInfC{$\Gamma \vdash path(\lambda y.\,at(\lambda x.\,A, a, a', p, y)) \deq p$}
\DisplayProof
\end{center}

\smallskip
\begin{center}
\AxiomC{$\Gamma \vdash p : Path(\lambda x.\,A, a, a')$}
\UnaryInfC{$\Gamma \vdash at(\lambda x.\,A, a, a', p, left) \deq a$}
\DisplayProof
\quad
\AxiomC{$\Gamma \vdash p : Path(\lambda x.\,A, a, a')$}
\UnaryInfC{$\Gamma \vdash at(\lambda x.\,A, a, a', p, right) \deq a'$}
\DisplayProof
\end{center}

We will often omit the first three arguments of $at$ since it is easy to infer them from the type of the fourth argument.

There is an obvious morphism $f : HPath \to Path$ such that $f(a \idtype a') = Path(\lambda x.\,ty(a), a, a')$.

The theory we are describing has many similarities to the theory of cubical sets.
The reason is that we can think of contexts $I, \ldots I \vdash$ as $n$-dimensional cubes.
Let $M$ be a model of $I$ and let $A,B \in M_{(ty,0)}$, then the sequence of sets
$\{\,x \in M_{(tm,n+1)}\ |\ ty(x) = (A, I, \ldots I \vdash B\!\uparrow^{n+1})\,\}$ has a natural structure of a cubical set.
If $M$ is a model of $coe_1 + Path$, then these cubical sets are fibrant, that is have fillers for all cubical horns.
We will formally define operations $Fill^n$ which give us these fillers in subsection~\ref{sec:fillers}.
Now we need these fillers to define several operations that we will use in the next subsection.

First, let us define $sq_l$ which satisfies all of the axioms of $sq$ together with axiom $\Gamma \vdash sq(i,right) \deq i$.
This operation is analogous to connections in cubical sets.
Actually, this construction shows that cubical sets that we defined before from a model of the theory have connections.
We can define $sq_l$ by filling the following horn:
\[ \xymatrix @C=0.5pc @R=0.5pc
    { left \ar[rrrr] \ar[dddd] &          & &                      & left \ar[dddd] \\
           & left \ar[rr] \ar[dd] \ar[ul] & & left \ar[dd] \ar[ur] & \\
           &                              & &                      & \\
           & left \ar[rr] \ar[dl]         & & left \ar[dr]         & \\
      left \ar[rrrr]           &          & &                      & right
    }\]
The inner, left, and top squares are $\lambda i\,j.\,left$, the bottom and right squares are $sq$,
and the filler gives us the outer square which is the required operation $sq_l$.
Formally, we define $sq_l(i,j)$ as
\[ at(at(coe_0(\lambda x_1.\,Path(\lambda x_2.\,left \idtype sq(x_1,x_2), refl(left), p_1), p_2),i),j) \]
where $p_1 = path(\lambda x_3.\,sq(x_1,x_3))$, $p_2 = refl(refl(left))$.

Operation $sq_r$ is similar to $sq_l$; it satisfies the following axioms:
\begin{align*}
\Gamma & \vdash sq_r(left,j) \deq j \\
\Gamma & \vdash sq_r(right,j) \deq right \\
\Gamma & \vdash sq_r(i,left) \deq i \\
\Gamma & \vdash sq_r(i,right) \deq right
\end{align*}
We can define $sq_r$ by filling the following horn:
\[ \xymatrix @C=0.5pc @R=0.5pc
    { left \ar[rrrr] \ar[dddd] &          & &                      & right \ar[dddd] \\
           & left \ar[rr] \ar[dd] \ar[ul] & & left \ar[dd] \ar[ur] & \\
           &                              & &                      & \\
           & left \ar[rr] \ar[dl]         & & left \ar[dr]         & \\
      right \ar[rrrr]           &         & &                      & right
    }\]
The inner square is $\lambda x_1\,x_2.\,left$, the left square is $\lambda x_2\,x_3.\,sq_l(x_2,x_3)$,
the top square is $\lambda x_1\,x_3.\,sq_l(x_1,x_3)$, the right square is $\lambda x_2.\,x_3.\,x_3$,
and the bottom square is $\lambda x_1\,x_3.\,x_3$.
The outer square gives us the required operation $sq_r$.

We will also need operation $dc'$ which satisfies the following axioms:
\begin{align*}
\Gamma & \vdash dc'(i,j,left) \deq i \\
\Gamma & \vdash dc'(i,j,right) \deq j \\
\Gamma & \vdash dc'(left,left,k) \deq left \\
\Gamma & \vdash dc'(right,right,k) \deq right
\end{align*}
Thus we need to find a map from $I^3$ to $I$, and we can do this by filling some horn.
Conditions that we put on $dc'$ are not enough to define a cubical horn, but we can fill missing parts.
Consider the following picture:
\[ \xymatrix @C=0.5pc @R=0.5pc
    { right \ar[rrrr] \ar[dddd] &         & &                      & right \ar[dddd] \\
           & left \ar[rr] \ar[dd] \ar[ul] & & left \ar[dd] \ar[ur] & \\
           &                              & &                      & \\
           & left \ar[rr] \ar[dl]         & & right \ar[dr]        & \\
      left \ar[rrrr]           &          & &                      & right
    }\]
Here $j$ is going from left to right, $k$ is going from top to bottom, and $i$ is going diagonally.
Top square is $\lambda j\,i.\,i$, bottom square is $\lambda j\,i.\,j$, left side of the inner square is $\lambda k.\,left$,
and the right side of the outer square is $\lambda k.\,right$.
We can take the inner square to be $\lambda j\,k.\,sq_l(j,k)$ and the right square to be $\lambda j\,k.\,sq_r(j,k)$.
The left square we can define by the filler operation.

For every $\Gamma \vdash a : A$, $\Gamma \vdash a' : A$, we have a type of 1-dimensional cubes (that is paths) between $a$ and $a'$.
We could also consider the type of $n$-dimensional cubes with given boundary.
For $n = 2$ this can be described as follows.
Suppose that we have terms $\Gamma, x : I \vdash p_{-0} : A$, $\Gamma, x : I \vdash p_{-1} : A$, $\Gamma, y : I \vdash p_{0-} : A$ and $\Gamma, y : I \vdash p_{1-} : A$
such that $p_{-0}[x \repl left] = p_{0-}[y \repl left] = a_{00}$, $p_{-0}[x \repl right] = p_{1-}[y \repl left] = a_{10}$,
$p_{0-}[y \repl right] = p_{-1}[x \repl left] = a_{01}$ and $p_{1-}[y \repl right] = p_{-1}[x \repl right] = a_{11}$.
\[ \xymatrix{ a_{00} \ar[r]^{p_{-0}} \ar[d]_{p_{0-}} & a_{10} \ar[d]^{p_{1-}} \\
              a_{01} \ar[r]_{p_{-1}}                 & a_{11}
            } \]
Then we define type $Square(p_{-0},p_{-1},p_{0-},p_{1-})$ of 2-dimensional cubes as
\[ Path(\lambda x.\,p_{-0} \idtype p_{-1}, path(\lambda y.\,p_{0-}), path(\lambda y.\,p_{1-})). \]

We can analogously define types of $n$-dimensional cubes for all $n$.
It is difficult to describe such types without heterogeneous path types,
but we can do this at least for small $n$.
For example, for $n = 2$ we can define it as either $p_{-0} * p_{1-} \idtype p_{0-} * p_{-1}$ (where $*$ is a concatenation of paths) or $p_{0-} \idtype p_{-0} * p_{1-} * sym(p_{-1})$.
It is easy to see that these types are (homotopy) equivalent, that is we can define mutually inverse functions between them.
We can show that they are also equivalent to $Square(p_{-0},p_{-1},p_{0-},p_{1-})$:

\begin{lem}[squares-eq]
Types $Square(p_{-0},p_{-1},p_{0-},p_{1-})$ and $p_{0-} \idtype p_{-0} * p_{1-} * sym(p_{-1})$ are equivalent.
\end{lem}
\begin{proof}
Let $*_l$ be a concatenation of paths such that $refl(x) *_l p \idtype p$ and let $*_r$ be a concatenation such that $p *_r refl(y) \idtype p$.
Since all concatenations are equivalent, we can replace $*$ with either of these operations.
We construct a type $\Gamma, i : I \vdash H$ such that $H[left] = Square(p_{-0},p_{-1},p_{0-},p_{1-})$ and $H[right] = (p_{0-} \idtype p_{-0} *_l p_{1-} *_r sym(p_{-1}))$.
Let $H$ be equal to
\[ Path(\lambda x.\,at(p_{-0},sq_l(x,i)) \idtype at(p_{-1},sq_l(x,i)), path(\lambda y.\,p_{0-}), t), \]
where $t$ equals to
\[ path(\lambda j.\,at(p_{-0},sq_r(i,j))) *_l path(\lambda y.\,p_{1-}) *_r sym(path(\lambda j.\,at(p_{-1},sq_r(i,j)))). \]
Then $H$ satisfies the required conditions.
It is easy to define an equivalence between $H[left]$ and $H[right]$.
\end{proof}

\subsection{Local versions of $coe$}

Usually, we can define two different versions of an eliminator for a type in type theory, one of which is stronger.
For example, stronger versions of $coe$ look like this:

\medskip
\begin{center}
\AxiomC{$\Gamma, x : I, \Delta \vdash D\ type$}
\AxiomC{$\Gamma, \Delta[x \repl left] \vdash d : D[x \repl left]$}
\BinaryInfC{$\Gamma, \Delta[x \repl right] \vdash coe^l_0(\lambda x.\,D, d) : D[x \repl right]$}
\DisplayProof
\end{center}

\medskip
\begin{center}
\AxiomC{$\Gamma, x : I, \Delta \vdash D\ type$}
\AxiomC{$\Gamma, \Delta[x \repl left] \vdash d : D[x \repl left]$}
\AxiomC{$\Gamma \vdash i : I$}
\TrinaryInfC{$\Gamma, \Delta[x \repl i] \vdash coe^l_1(\lambda x.\,D, d, i) : D[x \repl i]$}
\DisplayProof
\end{center}

\medskip
\begin{center}
\AxiomC{$\Gamma, x : I, \Delta \vdash D\ type$}
\AxiomC{$\Gamma, \Delta[x \repl left] \vdash d : D[x \repl left]$}
\BinaryInfC{$\Gamma, \Delta[x \repl left] \vdash coe^l_1(\lambda x.\,D, d, left) \deq d$}
\DisplayProof
\end{center}

\medskip
\begin{center}
\AxiomC{$\Gamma, x : I, \Delta \vdash D\ type$}
\AxiomC{$\Gamma \vdash i : I$}
\AxiomC{$\Gamma, \Delta[x \repl i] \vdash d : D[x \repl i]$}
\AxiomC{$\Gamma \vdash j : I$}
\QuaternaryInfC{$\Gamma, \Delta[x \repl j] \vdash coe^l_2(\lambda x.\,D, i, d, j) : D[x \repl j]$}
\DisplayProof
\end{center}

\medskip
\begin{center}
\AxiomC{$\Gamma, x : I, \Delta \vdash D\ type$}
\AxiomC{$\Gamma, \Delta[x \repl left] \vdash d : D[x \repl left]$}
\BinaryInfC{$\Gamma, \Delta[x \repl left] \vdash coe^l_2(\lambda x.\,D, left, d, left) \deq d$}
\DisplayProof
\end{center}

We can also consider theory $coe^l_2 + \beta^l_2$ which is $coe^l_2$ together with the following axiom:
\begin{center}
\AxiomC{$\Gamma, x : I, \Delta \vdash D\ type$}
\AxiomC{$\Gamma \vdash i : I$}
\AxiomC{$\Gamma, \Delta[x \repl i] \vdash d : D[x \repl i]$}
\TrinaryInfC{$\Gamma, \Delta[x \repl i] \vdash coe^l_2(\lambda x.\,D, i, d, i) \deq d$}
\DisplayProof
\end{center}

We call such versions of these operations \emph{local}, and the ones that were defined before \emph{global}.
The relationship between local versions of these operations is the same as between global ones.
Usually, if we have $\Pi$ type, then we can define local versions in terms of global, but without them local are strictly stronger.
But this is not the case for $coe_2 + \beta_2$; it turns out that $coe^l_2 + \beta^l_2$ follow from $coe_2 + \beta_2$ even without $\Pi$ types.

It is not convenient to work with such local operations directly since context in the conclusion is extended, but we can always rewrite them in the usual form.
For example, if $\Delta$ equals to $y_1 : B_1, \ldots y_k : B_k$, then we can rewrite $coe^l_2$ as follows:
\medskip
\begin{center}
\def\extraVskip{1pt}
\Axiom$\fCenter \Gamma, x : I, \Delta \vdash D\ type$
\noLine
\UnaryInf$\fCenter \Gamma \vdash i : I$
\noLine
\UnaryInf$\fCenter \Gamma, \Delta[x \repl i] \vdash d : D[x \repl i]$
\noLine
\UnaryInf$\fCenter \Gamma \vdash j : I$
\def\extraVskip{2pt}
\Axiom$\fCenter \Gamma \vdash b_1 : B_1[x \repl j]$
\noLine
\UnaryInf$\fCenter \ldots$
\noLine
\UnaryInf$\fCenter \Gamma \vdash b_k : B_k[x \repl j, y_1 \repl b_1, \ldots y_{k-1} \repl b_{k-1}]$
\BinaryInfC{$\Gamma \vdash coe^{l'}_2(\lambda x\,y_1 \ldots y_k.\,D, i, d, j, b_1, \ldots b_k) : D[x \repl j, y_1 \repl b_1, \ldots y_k \repl b_k]$}
\DisplayProof
\end{center}
Then theories $coe^l_2$ and $coe^{l'}_2$ are isomorphic.
Maps between them are defined as follows:
\begin{align*}
coe^l_2(\lambda x.\,D, i, d, j) & = coe^{l'}_2(\lambda x\,y_1 \ldots y_k.\,D, i, d, j, y_1, \ldots y_k) \\
coe^{l'}_2(\lambda x\,y_1 \ldots y_k.\,D, i, d, j, b_1, \ldots b_k) & = coe^l_2(\lambda x.\,D, i, d, j)[b_1, \ldots b_k]
\end{align*}

Now, we can define a map $coe^{l'}_2 + \beta^l_2 \to coe_2 + \beta_2$.
First, let $b'_m(z)$ be equal to
\[ coe_2(\lambda x.\,B_m[y_1 \repl b_1'(x), \ldots y_{m-1} \repl b_{m-1}'(x)], j, b_m, z) \]
for every $1 \leq m \leq k$.
If $\Gamma \vdash z : I$, then $\Gamma \vdash b_m'(z) : B_m[x \repl z, y_1 \repl b_1'(z), \ldots b_{m-1}'(z)]$.
Now, we can define $coe^{l'}_2(\lambda x\,y_1 \ldots y_k.\,D, i, d, j, b_1, \ldots b_m)$ as follows:
\[ coe_2(\lambda x.\,D[y_1 \repl b_1'(x), \ldots y_k \repl b_k'(x)], i, d[y_1 \repl b'_1(i), \ldots y_k \repl b'_k(i)], j) \]

Theories $coe^{l'}_0$ and $coe^{l'}_1$ are defined similarly to $coe^{l'}_2$.
Theory $coe^{l'}_1 + \sigma^l$ is $coe^{l'}_1$ together with the following axiom:
\medskip
\begin{center}
\def\extraVskip{1pt}
\Axiom$\fCenter \Gamma, \Delta \vdash D\ type$
\noLine
\UnaryInf$\fCenter \Gamma, \Delta \vdash d : D$
\noLine
\UnaryInf$\fCenter \Gamma \vdash i : I$
\def\extraVskip{2pt}
\Axiom$\fCenter \Gamma \vdash b_1 : B_1$
\noLine
\UnaryInf$\fCenter \ldots$
\noLine
\UnaryInf$\fCenter \Gamma \vdash b_k : B_k[y_1 \repl b_1, \ldots y_{k-1} \repl b_{k-1}]$
\BinaryInfC{$\Gamma \vdash coe^{l'}_1(\lambda x\,y_1 \ldots y_k.\,D, d, i, b_1, \ldots b_k) \deq d[b_1, \ldots b_k]$}
\DisplayProof
\end{center}

If we have heterogeneous path types, then we can define $coe^{l'}_0 + \sigma^l + sq$ in terms of $coe_0 + \sigma + sq$ and $coe^{l'}_1 + \sigma^l$ in terms of $coe_1 + \sigma$.
We can define map $coe^{l'}_1 \to coe^{l'}_0$ as before:
\[ coe^{l'}_1(\lambda x\,y_1 \ldots y_k.\,D, d, i, b_1, \ldots b_k) = coe^{l'}_0(\lambda x\,y_1 \ldots y_k.\,D[x \repl sq(x,i)], d, b_1, \ldots b_k) \]
Since we already know that there are maps going in both directions between $coe_0 + \sigma + sq$ and $coe_1 + \sigma$,
we just need to construct a map $coe^{l'}_0 + \sigma^l + sq + Path \to coe_0 + \sigma + sq + Path$.

To do this, first we define a map $coe_2 + \beta_1 \to coe_0 + \sigma + sq + Path$:
\[ coe_2(\lambda x.\,D, i, d, j)  = coe_0(\lambda x.\,D[x \repl dc'(i,j,x)], d) \]
Now, we can define $coe^{l'}_0(\lambda x\,y_1 \ldots y_k.\,D, d, b_1, \ldots b_m)$ as follows:
\[ coe_0(\lambda x.\,D[y_1 \repl b_1'(x), \ldots y_k \repl b_k'(x)], d[y_1 \repl b'_1(left), \ldots y_k \repl b'_k(left)]) \]
where $b'_m(z)$ equals to
\[ coe_2(\lambda x.\,B_m[y_1 \repl b_1'(x), \ldots y_{m-1} \repl b_{m-1}'(x)], right, b_m, z) \]

\subsection{Fillers}
\label{sec:fillers}

We already saw examples of two and three-dimensional filler operations.
Here we define theories $Fill^l_{tm}$ and $Fill_{tm}$ of local and global filler operations.
We will also construct morphisms between $Fill_{tm} + Path$ and $coe_1 + Path$;
It might be possible to construct a map from $Fill^l_{tm}$ to $coe_2 + \beta_2 + Path$, but it is more complicated and we will not need this construction.
In this subsection we switch back to the presentation of terms using De Bruijn indices since it will be notationally more convenient.

First, let us introduce a bit of notation.
Recall that if $(a_1, \ldots a_k)$ is a morphism of contexts $\Gamma$ and $\Delta$ and $b : (p,k+m)$ is such that $\vdash ctx^m(b) \deq \Delta$,
then we have $s = subst^m(\Gamma, b, a_1, \ldots a_k) : (p,n+m)$ such that $\vdash ctx^m(s) \deq \Gamma$.
We will also denote $s$ by $(a_1, \ldots a_k)^*(b)$.
In particular, if $h : (p,n+1)$ is such that $\vdash ctx^n(h) \deq I$, then for every $c \in \{ left, right \}$, we have $c^*(h) = subst^n(\emptyCtx, h, c) : (ctx,n)$.
Also, if $a : (p,n)$, then let $I \times a = subst^n(I, a) : (p,n+1)$.

Now, we need to describe certain morphisms of contexts which corresponds to cubical faces.
For every $n,k \in \mathbb{N}$, $0 \leq i \leq k$, and $c \in \{ left, right \}$, let $[i = c]$ denote sequence $v_{n+k-1}, \ldots v_i, c, v_{i-1}, \ldots v_0$.
The idea is that if $\Gamma : (ctx,n)$, $a : (p,n+k+1+m)$, and $ctx^m(a) = (\Gamma, I^{k+1} \vdash)$,
then $[i = c]^*(a)$ corresponds to the left or right (depending on $c$) $i$-th face of $a$.
We will also need an operation that gives us degenerate cubes.
For every $\Gamma : (ctx,n)$, $a : (p,n+k+m)$ and $0 \leq i \leq k$,
let $\delta_i(a) = subst^m((\Gamma, I^{k+1} \vdash), a, v_{n+k}, \ldots v_{i+1}, v_{i-1}, \ldots v_0)$.

Now we can define regular theory $Fill^l_{tm}$.
It is the regularization of a theory that has function symbol $Fill^n_{(tm,n+k)} : (ty,n+k) \times (tm,n-1+k)^{2n-1} \to (tm,n+k)$
for every $n,k \in \mathbb{N}$, $n > 0$, and the following axioms:
\medskip
\begin{center}
\def\extraVskip{1pt}
\Axiom$\fCenter I^n, \Delta \vdash D\ type$
\noLine
\UnaryInf$\fCenter I^{n-1}, [i = c]^*(\Delta) \vdash d_{[i = c]} : [i = c]^*(D)$
\noLine
\UnaryInf$\fCenter [i_1 = c]^*(d_{[i_2 = c']}) = [i_2-1 = c']^*(d_{[i_1 = c]}),\ 0 \leq i_1 < i_2 \leq n-1$
\def\extraVskip{2pt}
\UnaryInf$\fCenter I^n, \Delta \vdash Fill^n_{(tm,n+k)}(D, d_{[0 = left]}, d_{[0 = right]}, \ldots d_{[n-1 = left]}) : D$
\DisplayProof
\end{center}
where $Fill^n_{(tm,n+k)}$ has arguments of the form $d_{[i = c]}$ for every $0 \leq i < n$ and $c \in \{ left, right \}$ except for $d_{[n-1 = right]}$.

\medskip
\begin{center}
\def\extraVskip{1pt}
\Axiom$\fCenter I^n, \Delta \vdash D\ type$
\noLine
\UnaryInf$\fCenter I^{n-1}, [i = c]^*(\Delta) \vdash d_{[i = c]} : [i = c]^*(D)$
\noLine
\UnaryInf$\fCenter [i_1 = c]^*(d_{[i_2 = c']}) = [i_2-1 = c']^*(d_{[i_1 = c]}),\ 0 \leq i_1 < i_2 \leq n-1$
\def\extraVskip{2pt}
\UnaryInf$\fCenter [i = c]^*(Fill^n_{(tm,n+k)}(D, d_{[0 = left]}, d_{[0 = right]}, \ldots d_{[n-1 = left]})) = d_{[i = c]}$
\DisplayProof
\end{center}
\medskip

Theory $Fill_{tm}$ is the subtheory of $Fill^l_{tm}$ which has only function symbols of the form $Fill^n_{(tm,n+0)}$.
We will denote such function symbols by $Fill_{(tm,n)}$.

It is easy to define a map from $coe^l_1$ to $Fill^l_{tm}$:
\[ coe^l_1(D, d, i) = i^*(Fill^1(D, d)) \]
We also can define $Fill^1_{(tm,1+k)}(D, d_{[0 = left]})$ in terms of $coe^l_1$ as $coe^l_1(I \times D, I \times d_{[0 = left]}, v_k)$.
Actually, theory $coe^l_1$ and subtheory of $Fill^l_{tm}$ which consists of $Fill^1_{(tm,1+k)}$ are isomorphic.

We can define analogous maps between $Fill_{(tm,1)}$ and $coe_1$, but we also can define a morphism $Fill_{tm} \to coe_1 + Path$.
We define terms $Fill_{(tm,n)}$ by induction on $n$.
We already defined such term for $n = 1$, and we can define $Fill_{(tm,n+1)}$ as
\[ at(Fill_{(tm,n)}(Path(D, d_{[0 = left]}, d_{[0 = right]}), path(d_{[1 = left]}), \ldots path(d_{[n = left]}))\!\uparrow, v_0). \]

\subsection{Univalence}
\label{sec:univalence}

We will consider regular theory $wUA$ under $coe_0$, which has additional symbol
\[ iso : (ty,n)^2 \times (tm,n+1)^4 \times (tm,n) \to (ty,n) \]

Axioms of this theory have a lot of premises, so we list them now.
We will denote by $S$ the following set of formulae:
\begin{align*}
\Gamma & \vdash A\ type \\
\Gamma & \vdash B\ type \\
\Gamma, x : A & \vdash f : B \\
\Gamma, y : B & \vdash g : A \\
\Gamma, x : A, i : I & \vdash p : A \\
\Gamma, x : A & \vdash p[i \repl left] \deq g[y \repl f] \\
\Gamma, x : A & \vdash p[i \repl right] \deq x \\
\Gamma, y : B, i : I & \vdash q : B \\
\Gamma, y : B & \vdash q[i \repl left] \deq f[x \repl g] \\
\Gamma, y : B & \vdash q[i \repl right] \deq y
\end{align*}
If we have homogeneous path types, then the last six axioms can be replaced with the following two:
\begin{align*}
\Gamma, x : A & \vdash p : g[y \repl f] \idtype x \\
\Gamma, y : B & \vdash q : f[x \repl g] \idtype y
\end{align*}

Now, we can define axioms of $wUA$:
\medskip
\begin{center}
\AxiomC{$S$}
\AxiomC{$\Gamma \vdash j : I$}
\BinaryInfC{$\Gamma \vdash iso(A, B, \lambda x.\,f, \lambda y.\,g, \lambda x i.\,p, \lambda y i.\,q, j)\ type$}
\DisplayProof
\end{center}

\medskip
\begin{center}
\AxiomC{$S$}
\UnaryInfC{$\Gamma \vdash iso(A, B, \lambda x.\,f, \lambda y.\,g, \lambda x i.\,p, \lambda y i.\,q, left) \deq A$}
\DisplayProof
\end{center}

\medskip
\begin{center}
\AxiomC{$S$}
\UnaryInfC{$\Gamma \vdash iso(A, B, \lambda x.\,f, \lambda y.\,g, \lambda x i.\,p, \lambda y i.\,q, right) \deq B$}
\DisplayProof
\end{center}

\medskip
\begin{center}
\AxiomC{$S$}
\UnaryInfC{$\Gamma \vdash coe_0(\lambda j.\,iso(A, B, \lambda x.\,f, \lambda y.\,g, \lambda x i.\,p, \lambda y i.\,q, j), a) \deq f[x \repl a]$}
\DisplayProof
\end{center}
\medskip

This theory is similar to the univalence axiom, but it is defined for all types.
Actually, it seems that it is weaker than ordinary univalence, therefore we call this theory weak univalence.
We can add some additional rules to get the full univalence axiom, but this version will suffice for our purposes.
The (weak) univalence axiom for a universe follows from the assumption that this universe is closed under $iso$.

Although rules of $wUA$ do not imply that equivalences and paths between types are equivalent, we still can show that they are related.
First, we need to define a theory of equivalences.
Several equivalent definitions of equivalences are given in \cite{hottbook},
but some of them require additional constructions such as $\Sigma$ or $\Pi$ types.
Thus we will use a definition which requires only path types (actually, we can formulate it in such a way that $I$ will suffice).
Theory $Eq$ have the following axioms:
\begin{align*}
\Gamma & \vdash A\ type \\
\Gamma & \vdash B\ type \\
\Gamma, x : A & \vdash b : B \\
\Gamma, y : B & \vdash a_1 : A \\
\Gamma, x : A & \vdash p : a_1[y \repl b] \idtype x \\
\Gamma, y : B & \vdash a_2 : A \\
\Gamma, y : B & \vdash q : b[x \repl a_2] \idtype y
\end{align*}

Theory $UEq$ have one axiom $\Gamma, i : I \vdash H\ type$.
There is a canonical morphism $\varphi : Eq + coe_2 + \beta_1 \to UEq + coe_2 + \beta_2$.
To define it, let us first introduce auxiliary terms: $f(k) = coe_2(\lambda i.\,H, left, x, k)$ and $g(k) = coe_2(\lambda i.\,H, right, y, k)$.
If $\Gamma \vdash k : I$, then $\Gamma, x : H[i \repl left] \vdash f(k) : H[i \repl k]$ and $\Gamma, y : H[i \repl right] \vdash g(k) : H[i \repl k]$.
Now, we can define $\varphi$ as follows:
\begin{align*}
\varphi(A) & = H[i \repl left] \\
\varphi(B) & = H[i \repl right] \\
\varphi(b) & = f(right) \\
\varphi(a_1) & = g(left) \\
\varphi(a_2) & = g(left) \\
\varphi(p) & = coe_2(\lambda j.\,coe_2(\lambda i.\,H, j, f(j), left) \idtype x, left, refl(x), right) \\
\varphi(q) & = coe_2(\lambda j.\,coe_2(\lambda i.\,H, j, g(j), right) \idtype y, right, refl(x), left)
\end{align*}

A theory $UA$ of univalence should satisfy condition that $\varphi + id_{UA} : Eq + coe_2 + \beta_1 + UA \to UEq + coe_2 + \beta_1 + UA$ is an equivalence of theories (in some sense).
In the case of $wUA$, we still can construct a map $\psi : UEq + coe_2 + \beta_1 + wUA \to Eq + coe_2 + \beta_1 + wUA$ such that $\psi \circ \varphi'$ is homotopic to $id_{Eq + coe_2 + \beta_1 + wUA}$, where $\varphi' = \varphi + id_{wUA}$.
This means that for every symbol $x$ of $Eq$ such that $\Delta \vdash x$ we can define a term $h(x)$ in $Eq$ such that
$\Delta, i : I \vdash h(x)$, axioms of $Eq$ hold, $\Delta \vdash h(x)[left] \deq \psi(\varphi(x))$ and $\Delta \vdash h(x)[right] \deq x$.

\begin{lem}[UA]
There exists a map $\psi : UEq + coe_2 + \beta_1 + wUA \to Eq + coe_2 + \beta_1 + wUA$ such that $\psi \circ \varphi'$ is homotopic to $id$.
\end{lem}
\begin{proof}
Let $\psi(H) = iso(A, B, \lambda x.\,b, \lambda y.\,a_1, \lambda x.\,p, \lambda y.\,q_1, i)$, where
$q_1$ is the concatenation of paths $path(\lambda j.\,b[x \repl at(a_1, a_2, pa,j)])$ and $q$,
$pa$ is a path between $a_1$ and $a_2$, which can be obtained from $p$ and $q$ as the following concatenation:
\[ a_1 = a_1[y \repl y] \idtype a_1[y \repl b[x \repl a_2]] = a_1[y \repl b][x \repl a_2] \idtype x[x \repl a_2] = a_2. \]

We can define $h(A)$ as $A$, $h(B)$ as $B$ and $h(b)$ as $b$.
Thus we only need to construct $h(a_1)$, $h(a_2)$, $h(p)$ and $h(q)$.
Terms $h(a_1)$ and $h(a_2)$ should be the following paths:
\begin{align*}
& h(a_1) : coe_2(\lambda i.\,\psi(H), right, y, left) \idtype a_1 \\
& h(a_2) : coe_2(\lambda i.\,\psi(H), right, y, left) \idtype a_2
\end{align*}
Terms $h(p)$ and $h(q)$ should be squares with the following boundaries:
\[ \xymatrix{ coe_2(\lambda i.\,\psi(H), right, b, left) \ar[rr]^-{\psi(\varphi'(p))} \ar[d]_{h(a_1)[y \repl b]} & & x \ar[d]^{refl(x)} \\
              a_1[y \repl b] \ar[rr]_-p                                                                          & & x
            } \]
\[ \xymatrix{ b[x \repl coe_2(\lambda i.\,\psi(H), right, y, left)] \ar[rr]^-{\psi(\varphi'(q))} \ar[d]_{b[x \repl h(a_2)]} & & y \ar[d]^{refl(y)} \\
              b[x \repl a_2] \ar[rr]_-q                                                                                     & & y
            } \]
We construct $h(a_1)$ and $h(p)$, the other two terms are constructed analogously.
First, note that to construct a square with boundary given on the left, it is enough to construct a square with boundary given on the right:
\[ \xymatrix{ a_{00} \ar[r]^{p_{-0}} \ar[d]_{p_{0-}} & a_{10} \ar[d]^{p_{1-}} \\
              a_{01} \ar[r]_{p_{-1}}                 & a_{11}
            }
\qquad
   \xymatrix{ a_{00}[x \repl a_1] \ar[rr]^{p_{-0}[x \repl a_1]} \ar[d]_{p_{0-}[x \repl a_1]} & & a_{10}[x \repl a_1] \ar[d]^{p_{1-}[x \repl a_1]} \\
              a_{01}[x \repl a_1] \ar[rr]_{p_{-1}[x \repl a_1]}                              & & a_{11}[x \repl a_1]
            } \]
Indeed, if $T$ is a filler for the square on the right, then we can construct the following diagram:
\[ \xymatrix{ a_{00} \ar[rr]^-{a_{00}[x \repl sym(p)]} \ar[d]_{p_{0-}} & & a_{00}[x \repl a_1'] \ar[rr]^{p_{-0}[x \repl a_1']} \ar[d]_{p_{0-}[x \repl a_1']} & & a_{10}[x \repl a_1'] \ar[d]_{p_{1-}[x \repl a_1']} \ar[rr]^-{a_{10}[x \repl p]} & & a_{10} \ar[d]_{p_{1-}} \\
              a_{01} \ar[rr]_-{a_{01}[x \repl sym(p)]}                 & & a_{01}[x \repl a_1'] \ar[rr]_{p_{-1}[x \repl a_1']}                               & & a_{11}[x \repl a_1']                               \ar[rr]_-{a_{10}[x \repl p]} & & a_{11}
            } \]
where $a_1' = a_1[y \repl b]$, the middle square is $T[y \repl b]$ and side squares are naturality squares.
The right one, for example, is defined as $p_{1-}[x \repl p]$.
\Rlem{squares-eq} implies that these squares commute (up to homotopy).
By naturality and \rlem{squares-eq}, top and bottom rows are homotopic to $p_{-0}$ and $p_{-1}$ respectively.
Thus $p_{-0} * p_{1-} \idtype p_{0-} * p_{-1}$, and \rlem{squares-eq} implies that there exists a filler for the corresponding square.

Thus to construct $h(p)$, it is enough to define $h(a_1)$ in such a way that the outer square in the following diagram commutes:
\[ \xymatrix{ coe_2(\lambda i.\,\psi(H), right, y, left) \ar[r]^{q_1'} \ar[d]_{h(a_1)} & coe_2(\lambda i.\,\psi(H), right, b', left) \ar[rr]^-{\psi(\varphi'(p))[x \repl a_1]} \ar[d]_{h(a_1)[y \repl b[x \repl a_1]]} & & a_1 \ar[d]_{refl(a_1)} \\
              a_1 \ar[r]_{a_1[y \repl sym(q_1)]} & a_1[y \repl b'] \ar[rr]_-{p[x \repl a_1]}                                                                                                                           & & a_1
            } \]
where $b' = b[x \repl a_1]$ and $q_1' = coe_2(\lambda i.\,\psi(H), right, sym(q_1), left)$.
The left square commutes by naturality.
Since $q_1'$ has an inverse, the right square also commutes.
\end{proof}

\section{Model structure on models of theories with an interval type}
\label{sec:model-structure}

In this section for every regular theory $T$ (see \cite{alg-tt} for a definition of regular theories)
under $coe_1 + \sigma + Path + wUA$, we define a model structure on the category of models of $T$.
Every object of this model structure is fibrant and weak equivalences have several equivalent descriptions.

\subsection{Construction of models}

First, we need to describe several constructions of models of a theory.

For every model $M$ of a theory $T$, we define a theory $Lang(M)$.
It has function and predicate symbols of $T$ together with function symbol $O_a : s$ for every $a \in A_s$.
Axioms of $Lang(M)$ are axioms of $T$ together with the following sequents:
\begin{align*}
& \sststile{}{} O_a \downarrow \\
& \sststile{}{} \sigma(O_{a_1}, \ldots O_{a_k}) = O_{M(\sigma)(a_1, \ldots a_k)} \\
& \sststile{}{} R(O_{a_1}, \ldots O_{a_k})
\end{align*}
for every $a \in A_s$, every $a_i \in A_{s_i}$,
every $\sigma \in \mathcal{F}$ such that $M(\sigma)(a_1, \ldots a_k)$ is defined,
and every $R \in \mathcal{P}$ such that $(a_1, \ldots a_k) \in M(R)$.

Models of $Lang(M)$ are just models of $T$ together with a morphism from $M$.
That is, categories $M/\Mod{T}$ and $\Mod{Lang(M)}$ are isomorphic.
In particular, $A$ has a natural structure of a model of $Lang(M)$ defined as follows:
\begin{align*}
\alpha'(f)(O_a) & = a \\
\alpha'(f)(\sigma(x_1, \ldots x_k)) & = \alpha(f)(\sigma(x_1, \ldots x_k)) \\
\beta'(f)(R(x_1, \ldots x_k)) & = \beta(f)(R(x_1, \ldots x_k))
\end{align*}

\begin{lem}[cl-term]
If $t \in Term_\mathcal{F}(\varnothing)_s$ is such that $\sststile{}{} t \downarrow$ is a theorem of $Lang(M)$,
    then there is a unique $a \in A_s$ such that $\sststile{}{} t = O_a$ is a theorem of $Lang(M)$.
\end{lem}
\begin{proof}
Since $(A,\alpha',\beta')$ is a model of $Lang(M)$, for every theorem $\varphi \sststile{}{V} \psi$ of $Lang(M)$
    and every total function $f : V \to A$, if $\beta'(f)(\varphi) = \top$, then $\beta'(f)(\psi) = \top$.
In particular, if $\sststile{}{} O_a = O_{a'}$, then $a = a'$.
Hence if $\sststile{}{} t = O_a$ and $\sststile{}{} t = O_{a'}$, then $a = a'$, so such $a$ is unique.

Let us prove its existence.
We do this by induction on $t$.
If $t = O_a$, then we are done.
If $t = \sigma(t_1, \ldots t_k)$, then by induction hypothesis, $\sststile{}{} t = \sigma(O_{a_1}, \ldots O_{a_k})$ for some $a_1$, \ldots $a_k$.
Note that if $\sststile{}{} \sigma(O_{a_1}, \ldots O_{a_k})\!\!\downarrow$ is derivable, then $M(\sigma)(a_1, \ldots a_k)$ is defined.
Thus $\sststile{}{} \sigma(O_{a_1}, \ldots O_{a_k}) = O_{M(\sigma)(a_1, \ldots a_k)}$ is also derivable.
\end{proof}

For every morphism $h : M \to N$ of models of $T$, we can define a morphism $Lang(h) : Lang(M) \to Lang(N)$ of theories under $T$ as $Lang(h)(O_a) = O_{h(a)}$.
Thus $Lang$ is a functor $\Mod{T} \to T/\Th_\mathcal{S}$.

\begin{prop}[lang-ff]
$Lang$ is fully faithful.
\end{prop}
\begin{proof}
Let $h_1$, $h_2$ be morphisms of models such that $Lang(h_1) = Lang(h_2)$.
Then $O_{h_1(a)} = Lang(h_1)(O_a) = Lang(h_2)(O_a) = O_{h_2(a)}$, and by \rlem{cl-term}, $h_1(a) = h_2(a)$.
Thus $Lang$ is faithful.

Let $M_1 = (A_1,\alpha_1,\beta_1)$ and $M_2 = (A_2,\alpha_2,\beta_2)$ be models of $T$,
    and let $h : Lang(M_1) \to Lang(M_2)$ be a morphism of theories under $T$.
Then by \rlem{cl-term}, for every $a \in A_1$, there is a unique $h'(a) \in A_2$ such that $\sststile{}{} h(O_a) = O_{h'(a)}$ is a theorem of $Lang(M_2)$.
Let us show that $h' : A_1 \to A_2$ is a morphism of models $M_1$ and $M_2$.
Indeed, if $M_1(\sigma)(a_1, \ldots a_k)$ is defined, then $\sststile{}{} \sigma(O_{a_1}, \ldots O_{a_k}) = O_{M_1(\sigma)(a_1, \ldots a_k)}$ is a theorem of $Lang(M_1)$.
Hence \[ \sststile{}{} \sigma(O_{h'(a_1)}, \ldots O_{h'(a_k)}) = O_{h'(M_1(\sigma)(a_1, \ldots a_k))} \] is a theorem of $Lang(M_2)$.
But \[ \sststile{}{} \sigma(O_{h'(a_1)}, \ldots O_{h'(a_k)}) = O_{M_2(\sigma)(h'(a_1), \ldots h'(a_k))} \] is also a theorem of $Lang(M_2)$.
Hence by \rlem{cl-term}, $h'(M_1(\sigma)(a_1, \ldots a_k)) = M_2(\sigma)(h'(a_1), \ldots h'(a_k))$.

If $(a_1, \ldots a_k) \in M_1(R)$, then $\sststile{}{} R(O_{a_1}, \ldots O_{a_k})$ is a theorem of $Lang(M_1)$.
Hence $\sststile{}{} R(O_{h'(a_1)}, \ldots O_{h'(a_k)})$ is a theorem of $Lang(M_2)$.
Since $M_2$ is a model of $Lang(M_2)$, it follows that $(h'(a_1), \ldots h'(a_k)) \in M_2(R)$.

Thus $h'$ is a morphism of models.
Note that by definition of $h'$, $Lang(h') = h$.
Hence $Lang$ is full.
\end{proof}

Now, let us describe a functor $Syn : T/\Th_\mathcal{S} \to \Mod{T}$.
For every $i : T \to T'$, let $Syn(i) = i^*(0_{T'})$, where $0_{T'}$ is the initial object of $\Mod{T'}$,
and $i^* : \Mod{T'} \to \Mod{T}$ is the functor that was defined in \cite{alg-tt}.
If $f : T_1 \to T_2$ is a morphism of theories under $T$, then let $Syn(f) = i_1^*(!_{f^*(0_{T_2})})$,
where $!_{f^*(0_{T_2})}$ is the unique morphism $0_{T_1} \to f^*(0_{T_2})$.

The construction of initial models of partial Horn theories was given in \cite{PHL}.
Let us repeat it here.
Let $T = ((\mathcal{S},\mathcal{F},\mathcal{P}),\mathcal{A})$ be a standard partial Horn theory.
First, we define a partial equivalence relations on sets $Term_\mathcal{F}(\varnothing)$ as $t_1 \sim t_2$ if and only if $\sststile{}{} t_1 = t_2$ is a theorem of $T$.
The interpretation of $R \in \mathcal{P}$ consists of tuples $(t_1, \ldots t_k)$ such that $\sststile{}{} R(t_1, \ldots t_k)$ is derivable in $T$.
Then $\mathcal{S}$-set $Term_\mathcal{F}(\varnothing)/\!\sim$ has a natural structure of a model of $((\mathcal{S},\mathcal{F},\mathcal{P}),\mathcal{A})$, and this model is initial.

\begin{prop}[syn-lang]
$Syn$ is right adjoint to $Lang$.
\end{prop}
\begin{proof}
Let $\epsilon_{T'} : Lang(Syn(T')) \to T'$ be defined as $\epsilon_{T'}(O_t) = t$.
It is easy to see that $\epsilon_{T'}$ preserves axioms of $Lang(Syn(T'))$.
Moreover, $\epsilon$ is natural in $T'$.
Let us prove that $\epsilon$ is the counit of the adjunction.
Let $f : Lang(M) \to T'$ be a morphism.
Then we need to show that there is a unique morphism $g : Lang(M) \to Lang(Syn(T'))$ such that $\epsilon_{T'} \circ g = f$.
By \rlem{cl-term}, there is a unique $t$ such that $g(O_a) = O_t$.
Since $t = \epsilon_{T'}(g(O_a)) = f(O_a)$, $g$ must satisfy equation $g(O_a) = O_{f(O_a)}$.
Thus $g$ is unique.
It is easy to see that this $g$ preserves axioms of $Lang(M)$; hence it defines a morphism $g : Lang(M) \to Lang(Syn(T'))$.
\end{proof}

\begin{rem}[colimits]
Propositions \nprop{lang-ff} and \nprop{syn-lang} imply that colimits of models can be constructed as follows:
\[ colim_{j \in J}(M_j) = Syn(Lang(colim_{j \in J}(M_j))) = Syn(colim_{j \in J}(Lang(M_j))). \]
Since colimits of theories have simple explicit description (see \cite{alg-tt}), this gives us explicit description of colimits of models.
\end{rem}

For every morphism of theories $f : T \to T'$ there is a functor $f^* : \Mod{T'} \to \Mod{T}$ which was constructed in \cite{alg-tt}.
We also can define functor $f_! : \Mod{T} \to \Mod{T'}$ as $f_!(M) = Syn(Lang(M) \amalg_{T} T')$.
It was shown in \cite{PHL} that $f_!$ is left adjoint to $f^*$.
This theorem was proved there only for a weaker notion of morphisms of theories, but the proof also works for general morphisms as defined in \cite{alg-tt}.

Functor $f_!$ can be used to present a model of a theory by generators and relations.
Let $T$ be a fixed $\mathcal{S}$-theory.
Note that models of the empty theory are just $\mathcal{S}$-sets.
If $f : 0 \to T$ is the unique morphism from the empty theory, then $f^*(M)$ is just the underlying $\mathcal{S}$-set of $M$,
    and $f_!(X)$ is the free model of $T$ on $\mathcal{S}$-set $X$.
We will denote this free model by $F(X)$.
If $R$ is a set of axioms in the language of theory $Lang(X) \amalg T$,
    then let $F(X,R)$ be a model of $T$ defined as $Syn(Lang(X) \amalg T \cup R)$.
By definition of $Syn$, to construct a morphism $F(X,R) \to M$ it is necessary and sufficient
    to construct a morphism from $X$ to the underlying $\mathcal{S}$-set of $M$ such that relations from $R$ are true in $M$.

Sometimes we will omit the set of generators if it can be inferred from the set of relations.
For examples, we will write $F(\{\,\vdash p : Id(A,a,a')\,\})$ for the model $F(\{\,a : (tm,0), a' : (tm,0), A : (ty,0), p : (tm,0)\,\}, \{\,ty(p) = Id(A,a,a')\,\})$.
Another examples is $F(\{\,A_1, \ldots A_n \vdash a : A\,\})$ which equals to $F(\{\,A_i : (ty,i), A : (ty,n), a : (tm,n)\,\}, \{\,ty(a) = A, ft^{i+1}(A) = A_{n-i}\,\})$.
Thus this model is isomorphic to the free model $F(\{\,a : (tm,n)\,\})$.

\subsection{Model structure}

To construct a model structure on $\Mod{T}$, we need to recall a few definitions from \cite{f-model-structures}.
A reflexive path object $P(X)$ for an object $X$ is any factorization of the diagonal $X \to X \times X$.
A reflexive cylinder object $C_U(V)$ for a map $i : U \to V$ is any factorization of $[id_V,id_V] : V \amalg_U V \to V$.
Maps $f,g : V \to X$ are homotopic relative to a cylinder object $[\cyli_0,\cyli_1] : V \amalg_U V \to C_U(V)$, if there exists a map $h : C_U(V) \to X$
such that $h \circ \cyli_0 = f$ and $h \circ \cyli_1 = g$.
In this case we will write $f \sim_i g$.
We say that a map $f : X \to Y$ has RLP up to $\sim_i$ with respect to $i : U \to V$ if for every commutative square of the form
\[ \xymatrix{ U \ar[r]^u \ar@{}[dr]|(.7){\sim_i} \ar[d]_i & X \ar[d]^f \\
              V \ar[r]_v \ar@{-->}[ur]^g                  & Y,
            } \]
there is a dotted arrow $g : V \to X$ such that $g \circ i = u$ and $(f \circ g) \sim_i v$.

We will also need the following theorem from \cite{f-model-structures}:
\begin{thm}[model-structures]
Let $\C$ be a complete and cocomplete category, and let $\I$ be a set of maps of $\C$
such that the domains and the codomains of maps in $\I$ are cofibrant and small relative to $\Icell$.
For every $i : U \to V \in \I$, choose a reflexive relative cylinder object $C_U(V)$
such that $[\cyli_0,\cyli_1] : V \amalg_U V \to C_U(V) \in \Icof$.
Let $\J_\I = \{\ \cyli_0 : V \to C_U(V)\ |\ i : U \to V \in \I \ \}$, and
let $\we_\I$ be the set of maps which have RLP up to $\sim_i$ with respect to every $i \in \I$.

Suppose that for every object $X$, there exists a reflexive path object $P(X)$ such that the following conditions hold:
\begin{enumerate}
\item $p_0$ has RLP with respect to $\I$.
\item For every $f : X \to Y$, there exists a morphism of path objects $(f,P(f)) : P(X) \to P(Y)$,
\item For every object $X$, there exists a map $s : P(X) \to P(X)$ such that $p_0 \circ s = p_1$ and $p_1 \circ s = p_0$.
\item Either maps $\langle p_0, p_1 \rangle : P(X) \to X \times X$ have RLP with respect to $\J_\I$
or maps in $\Jcell[_\I]$ have RLP up to $\sim^{r*}$ with respect to the domains of maps in $\I$.
\end{enumerate}
Then there exists a cofibrantly generated model structure on $\C$ with $\I$ as a set of generating cofibrations,
$\J_\I$ as a set of generating trivial cofibrations, and $\we_\I$ as a class of weak equivalences.
\end{thm}
Here $\sim^{r*}$ denotes the reflexive transitive closure of the relation of right homotopy with respect to $P(X)$.

Let $\I_{tm}$ be the set of the following morphisms:
\[ i_{(tm,n)} : F(\{\,\Gamma \vdash A\ type\,\}) \to F(\{\,\Gamma \vdash a : A\,\}) \]
Let $\I_{ty}$ be the set of the following morphisms:
\[ i_{(ty,n)} : F(\{\,\Gamma \vdash ctx\,\}) \to F(\{\,\Gamma \vdash A\ type\,\}) \]
The set $\I$ of generating cofibrations is the union $\I_{tm} \cup \I_{ty}$.

For every $i : U \to V \in \I$, we need to define a relative cylinder object $C_U(V)$.
Let $C_{F(\{\,\Gamma \vdash A\ type\,\})}(F(\{\,\Gamma \vdash a : A\,\}))$ be equal to
\[ F(\{\,(\Gamma \vdash A\ type), (\Gamma, I \vdash h : A\!\uparrow)\,\}), \]
$\cyli_0(a) = h[left]$, $\cyli_1(a) = h[right]$, and let $s : C_U(V) \to V$ be defined as $s(h) = a\!\uparrow$.
Let $C_{F(\{\,\Gamma \vdash ctx\,\})}(F(\{\,\Gamma \vdash A\ type\,\})) = Syn(Eq + T)$, where $Eq$ is the theory defined in subsection~\ref{sec:univalence};
$\cyli_0(A) = A$, $\cyli_1(A) = B$, and let $s : C_U(V) \to V$ be defined as follows:
$s(A) = A$, $s(B) = A$, $s(\lambda x.\,b) = s(\lambda y.\,a_1) = s(\lambda y.\,a_2) = \lambda x.\,x$,
$s(\lambda x.\,p_1) = s(\lambda y.\,q_1) = s(\lambda x.\,p_2) = s(\lambda y.\,q_2) = \lambda x.\,refl(x)$.

Note that for every $i : U \to V \in \I$, $[\cyli_0,\cyli_1] : V \amalg_U V \to C_U(V)$ belongs to $\Icell[\I_{tm}]$.
If $i \in \I_{tm}$, then $[\cyli_0,\cyli_1]$ is (isomorphic to) $F(\{\,(\Gamma \vdash a : A), (\Gamma \vdash a' : A)\,\}) \to F(\{\,(\Gamma \vdash A), (\Gamma, I \vdash h : A\!\uparrow)\,\})$,
and this map is isomorphic to $F(\{\,(\Gamma \vdash a : A), (\Gamma \vdash : a' : A)\,\}) \to F(\{\,(\Gamma \vdash a : A), (\Gamma \vdash a' : A), (\Gamma \vdash p : a \idtype a')\,\})$,
which is obviously a pushout of a map from $\I_{tm}$.
If $i \in \I_{ty}$, then it is easy to see that $[\cyli_0,\cyli_1]$ is a composition of five maps which are pushouts of maps from $\I_{tm}$.

There is another class of cylinder objects for maps in $\I_{ty}$.
Let $C'_{F(\{\,\Gamma \vdash ctx\,\})}(F(\{\,\Gamma \vdash A\ type\,\}))$ be equal to $F(\{\,\Gamma, i : I \vdash H\ type\,\})$, $\cyli_0(A) = H[i \repl left]$, $\cyli_1(A) = H[i \repl right]$.
We cannot use these cylinder objects directly since $[\cyli_0,\cyli_1]$ is not a cofibration; nevertheless, they will be useful later.
We will denote the set of maps of the form $\cyli_0 : F(\{\,\Gamma \vdash A\ type\,\}) \to C'_{F(\{\,\Gamma \vdash ctx\,\})}(F(\{\,\Gamma \vdash A\ type\,\}))$ by $\J'_{\I_{ty}}$.

Now, let us describe a general definition of a functor $P : \Mod{T} \to \Mod{T}$ that works for every stable theory $(T,\alpha)$.
Let $P(X)_{(p,n)} = \{\,a \in P(X)_{(p,n+1)}\ |\ ctx^n(a) = I\,\}$.
For every function and predicate symbol $S$, define $P(X)(S)(a_1, \ldots a_k)$ as $X(\alpha(L(S)))(a_1, \ldots a_k)$.
Since $\alpha$ preserves theorems, this definition satisfies axioms of $T$; hence it is a correct definition of a model of $T$.
For every morphism of models $f : X \to Y$, let $P(f)(a) = f(a)$.
The fact that $f$ is a morphism of models implies that $P(f)$ is a morphism too.
It is obvious that $P$ preserves identity morphisms and compositions.

To define the structure of a path object on $P(X)$, we need to assume that $(T,\alpha)$ is regular.
In this case, we define $t : X \to P(X)$ as $t(a) = I \times a$, and $p_0,p_1 : P(X) \to X$ as $p_0(a) = left^*(a)$ and $p_1(a) = right^*(a)$.
The regularity condition ensures that function and predicate symbols are stable under operations $I \times -$ and $c^*(-)$.
Hence these definitions indeed determine morphisms of models.
The fact that $p_0 \circ t = p_1 \circ t = id_X$ follows from properties of operation $subst^n$.
If $T$ is under $HPath + coe_0$, then we can define $s : P(X) \to P(X)$ as $s(a) = subst^n(I, a, inv(v_0))$, where
\[ inv(i) = at(right, left, coe_0(v_0 \idtype left, refl(left)), i). \]

We will prove some of the conditions of \rthm{model-structures} in the following lemmas:
\begin{lem}[Jtm]
Maps $\langle p_0, p_1 \rangle : P(X) \to X \times X$ have RLP with respect to $\J_{\I_{tm}}$.
\end{lem}
\begin{proof}
We are given a type $x : I, \Delta \vdash A\ type$ and terms $x : I, \Delta \vdash f_1 : A$, $\Delta[x \repl left], y : I \vdash f_0 : A[x \repl left]$ and $\Delta[x \repl right], y : I \vdash f_2 : A[x \repl right]$,
and we need to find a term $t$ that satisfies $x : I, \Delta, y : I \vdash t : A$, $t[y \repl left] = f_1$, $t[x \repl left] = f_0$ and $t[x \repl right] = f_2$.
Thus, $t$ is just a local two-dimensional filler.
We cannot construct such fillers in general, but we can do it in our case since $A$ depends only on one of the coordinates.
Thus, the construction will be similar to the construction of $coe^l_1$.

Let $\Delta$ be equal to $z_1 : B_1, \ldots z_k : B_k$.
Then to construct a term $t$, we just need for every $\vdash i : I$, $\vdash j : I$ and $\vdash b_1 : B_1[x \repl i]$, \ldots $\vdash b_k : B_k[x \repl i, z_1 \repl b_1, \ldots z_{k-1} \repl b_{k-1}]$,
to find a term $t'(i, j, b_1, \ldots b_k)$ such that $\vdash t'(i, j, b_1, \ldots b_k) : A[x \repl i, y_1 \repl b_1, \ldots y_k \repl b_k]$, $t'(i, left, b_1, \ldots b_k) = f_1[x \repl i, y_1 \repl b_1, \ldots y_k \repl b_k]$,
$t'(left, j, b_1, \ldots b_k) = f_0[y \repl j, y_1 \repl b_1, \ldots y_k \repl b_k]$ and $t'(right, j, b_1, \ldots b_k) = f_2[y \repl j, y_1 \repl b_1, \ldots y_k \repl b_k]$.
Thus, $t'$ is an analog of $coe^{l'}_1$.

First, let $b'_m(w)$ be equal to
\[ coe_2(\lambda x.\,B_m[y_1 \repl b_1'(x), \ldots y_{m-1} \repl b_{m-1}'(x)], i, b_m, w) \]
Then the following conditions are satisfied:
\begin{align*}
& \vdash b'_m(w) : B_m[x \repl w, z_1 \repl b'_1(w), \ldots z_{m-1} \repl b'_{m-1}(w)] \\
x : I & \vdash f_1[\overline{z} \repl \overline{b'(x)}] : A[\overline{z} \repl \overline{b'(x)}] \\
y : I & \vdash f_0[\overline{z} \repl \overline{b'(left)}] : A[x \repl left, \overline{z} \repl \overline{b'(left)}] \\
y : I & \vdash f_2[\overline{z} \repl \overline{b'(right)}] : A[x \repl right, \overline{z} \repl \overline{b'(right)}]
\end{align*}
\end{proof}

If our theory has some additional structure (local fillers for types),
then we can prove that maps $\langle p_0, p_1 \rangle : P(X) \to X \times X$ have RLP with respect to $\J_{\I_{ty}}$,
but we cannot do this in general.
Thus we will use the second option and will prove that objects of the form $F(\{\,\Gamma : (ctx,n)\,\})$ have LLP up to $\sim^{r*}$ with respect to $\Jcell[_\I]$.

\begin{lem}[Jty]
Pushouts of maps from $\J'_{\I_{ty}}$ have RLP up to $\sim^r$ with respect to objects of the form $F(\{\,\Delta : (ctx,n)\,\})$.
\end{lem}
\begin{proof}
Maps from $F(\{\,\Delta : (ctx,n)\,\})$ to $Y$ may be identified with elements of $Y_{(ctx,n)}$.
Since $Y$ is a pushout $X \amalg_{F(\{\,\Gamma \vdash A\ type\,\})} F(\{\,\Gamma, I \vdash H\ type\,\})$,
by \rrem{colimits}, its elements can be described as closed terms of theory
$T' = Lang(X) \cup \{\,H : (ty,k+1), \sststile{}{} ctx(H) = I(ctx^2(H)) \land O_{u(A)} = H[left]\,\}$,
where $u$ is the map $F(\{\,\Gamma \vdash A\ type\,\}) \to X$.

For every set of variables $V$ and every term $t \in Term_{T'}(V)_{(p,n)}$,
we construct a term $h(t) \in Term_{T'}(L(V))_{(p,n+1)}$, where $L(V) = \{\,x : (p,k+1)\ |\ (x : (p,k)) \in V\,\}$.
\begin{align*}
h(x) & = x \text{, if } x \in Var \\
h(O_a) & = I \times O_a \\
h(\emptyCtx) & = I \\
h(ft(t')) & = ft(h(t')) \\
h(ty(t')) & = ty(h(t')) \\
h(\sigma(t_1, \ldots t_k)) & = I \times ctx(t) \vdash \sigma_1(h(t_1), \ldots h(t_k)) \\
h(H) & = I \times ctx(t) \vdash subst(H, v_k, \ldots v_1, sq(v_{k+1}, v_0))
\end{align*}
where $\sigma_1$ is the lift of $\sigma$ which is obtained from the stability of $T$.

For every formula $\varphi \in Form_{T'}(V)$, we can define formula $h(\varphi) \in Form_{T'}(L(V))$ as follows:
\begin{align*}
h(t_1 = t_2) & = (h(t_1) = h(t_2)) \\
h(R(t_1, \ldots t_k)) & = R_1(h(t_1), \ldots h(t_k)) \\
h(\varphi_1 \land \ldots \land \varphi_n) & = h(\varphi_1) \land \ldots \land h(\varphi_n)
\end{align*}
It is easy to see that $h$ is stable under substitution.
Thus to prove that for every theorem $\varphi \sststile{}{V} \psi$ of $T'$, $h(\varphi) \sststile{}{L(V)} h(\psi)$
is also a theorem of $T'$, it is enough to show that this is the case for axioms.
If a formula $\varphi$ does not mention $H$, then $h(\varphi)$ coincides with the lifting of $\varphi$.
Hence if a sequent $\varphi \sststile{}{V} \psi$ does not use $H$, then $h(\varphi) \sststile{}{L(V)} h(\psi)$ is a theorem by stability.
The only axiom that mentions $H$ is $\sststile{}{} ctx(H) = I(ctx^2(H)) \land O_{u(A)} = H[left]$.
It is easy to see that applying $h$ to this axiom produces a theorem.

If $t$ is closed term, then $h(t)$ is also closed.
Thus for every element $t \in Y_{(p,n)}$, we defined an element $h(t) \in Y_{(p,n+1)}$ such that $ctx^n(h(t)) = I$.
Thus $h(t)$ defines an element of $P(Y)_{(p,n)}$.
Moreover, $p_1 \circ h(t) = t$ and $p_0 \circ h(t)$ factors through $X$,
which shows that $X \to Y$ has RLP up to $\sim^r$ with respect to objects of the form $F(\{\,\Delta : (p,n)\,\})$.
\end{proof}

\begin{lem}[Jhom]
Let $u : F(\{\,\Gamma \vdash A\ type\,\}) \to X$ be a map.
Let $Y = X \amalg_{F(\{ \Gamma \vdash A\ type \})} C_U(V)$ and $Y' = X \amalg_{F(\{ \Gamma \vdash A\ type \})} C'_U(V)$.
Then there exist maps $v : Y \to Y'$ and $v' : Y' \to Y$ such that $v' \circ v \sim^r id_Y$.
\end{lem}
\begin{proof}
Maps $\varphi$ and $\psi$ that were defined in subsection~\ref{sec:univalence} induce maps $C_U(V) \to C'_U(V)$ and $C'_U(V) \to C_U(V)$, which induce maps $v$ and $v'$.
We can define a homotopy $h' : Y \to P(Y)$ between $v' \circ v$ and $id_Y$ as follows: $h'(y) = I \times y$ for every $y \in Y_{(p,k)}$
and $h'(x) = (i : I, \Delta \vdash subst(h(x), v_{n-1}, \ldots v_0, v_n))$ for every symbol $x$ of $Eq$,
where $h$ is a homotopy constructed in \rlem{UA}.
Note that $h'(A) = I \times A$, so $h'$ is a well-defined map by the universal property of pushouts.
\end{proof}

Now, we can complete the construction of the model structure:
\begin{thm}[main]
For every regular theory $T$ under $coe_1 + \sigma + Path + wUA$, sets $\I$ and $\J_\I$ and functor $P : \Mod{T} \to \Mod{T}$ satisfy conditions of \rthm{model-structures}.
\end{thm}
\begin{proof}
First, let us prove that $p_0 : P(X) \to X$ has RLP with respect to $\I$.
Indeed, given a type $I, \Delta \vdash H$ and a term $left^*(\Delta) \vdash a : left^*(H)$ in $X$,
we need to find a term $I, \Delta \vdash h : H$ such that $left^*(h) = a$.
We can define $h$ as $Fill^1_{(tm,1+n)}(H, a, v_n)$.
Given a context $I, \Delta$ and a type $left^*(\Delta) \vdash A$,
we need to find a type $I, \Delta \vdash H$ such that $left^*(H) = A$.
We can define $H$ as $subst(A, b_1, \ldots b_n)$, where $b_i = coe^l_2(A_i, v_n, v_{n-i}, left)$.

To prove the last condition, note that objects of the form $F(\{\,\Gamma : (ctx,n)\,\})$ are finite.
Thus it is enough to prove that they have LLP up to $\sim^{r*}$ with respect to pushouts of maps in $\J_\I$.
For pushouts of maps from $\J_{\I_{tm}}$, this follows from \rlem{Jtm}.
Let $f : X \to Y$ be a pushout of a map in $\J_{\I_{ty}}$ along a map $u : V \to X$ and let $f' : X \to Y'$ be the pushout of the corresponding map from $\J'_{\I_{ty}}$ along $u$.
Then \rlem{Jty} implies that for every $t : F(\{\,\Gamma : (ctx,n)\,\}) \to Y$, there exists a map $t' : F(\{\,\Gamma : (ctx,n)\,\}) \to X$ such that $f' \circ t' \sim^r v \circ t$.
By \rlem{Jhom}, $t \sim^r v' \circ v \circ t \sim^r v' \circ f' \circ t' = f \circ t'$.
Thus $t'$ is the required lifting.
\end{proof}

\subsection{Theories with sigma types}
\label{sec:sigma}

In this section we give several equivalent descriptions of weak equivalences between models of theories with $\Sigma$ types.
We also discuss the relationship between such models and fibration categories.

Recall that the theory of $\Sigma$ types with eta has the following rules:
\medskip
\begin{center}
\AxiomC{$\Gamma, A \vdash B\ type$}
\UnaryInfC{$\vdash \Sigma(A, B)\ type$}
\DisplayProof
\quad
\AxiomC{$\Gamma, A \vdash B\ type$}
\AxiomC{$\Gamma \vdash a : A$}
\AxiomC{$\Gamma \vdash b : B[a]$}
\TrinaryInfC{$\Gamma \vdash pair(A, B, a, b) : \Sigma(A, B)$}
\DisplayProof
\end{center}

\medskip
\begin{center}
\AxiomC{$\Gamma \vdash p : \Sigma(A, B)$}
\UnaryInfC{$\Gamma \vdash \pi_1(A, B, p) : A$}
\DisplayProof
\quad
\AxiomC{$\Gamma \vdash p : \Sigma(A, B)$}
\UnaryInfC{$\Gamma \vdash \pi_2(A, B, p) : B[\pi_1(A, B, p)]$}
\DisplayProof
\end{center}

\medskip
\begin{center}
\AxiomC{$\Gamma, A \vdash B\ type$}
\AxiomC{$\Gamma \vdash a : A$}
\AxiomC{$\Gamma \vdash b : B[a]$}
\TrinaryInfC{$\Gamma \vdash \pi_1(A, B, pair(A, B, a, b)) \deq a$}
\DisplayProof
\end{center}

\medskip
\begin{center}
\AxiomC{$\Gamma, A \vdash B\ type$}
\AxiomC{$\Gamma \vdash a : A$}
\AxiomC{$\Gamma \vdash b : B[a]$}
\TrinaryInfC{$\Gamma \vdash \pi_2(A, B, pair(A, B, a, b)) \deq b$}
\DisplayProof
\end{center}

\medskip
\begin{center}
\AxiomC{$\Gamma \vdash p : \Sigma(A, B)$}
\UnaryInfC{$\Gamma \vdash pair(A, B, \pi_1(A, B, p), \pi_2(A, B, p)) \deq p$}
\DisplayProof
\end{center}
We will usually omit first two arguments to $pair$, $\pi_1$ and $\pi_2$.

If $T$ has $\Sigma$ types and $X$ is a model of $T$, then every nonempty context of $X$ is isomorphic to a context of length 1.
Indeed, for every $(A_1, \ldots A_n \vdash) \in X_{(ty,n)}$, we define $\Sigma(A_1, \ldots A_n) \in X_{(ty,0)}$ as $\Sigma(A_1, \ldots \Sigma(A_{n-1},A_n))$.
Morphisms between $\Gamma$ and $\Sigma(\Gamma)$ are defined as follows:
$c : \Gamma \to \Sigma(\Gamma)$ is $pair(v_{n-1}, \ldots pair(v_1,v_0), \ldots))$ and $d : \Sigma(\Gamma) \to \Gamma$ is $b_0, \ldots b_{n-1}$,
where $b_{n-1} = \pi_2(\ldots \pi_2(v_0) \ldots )$, where $\pi_2$ is repeated $n-1$ times, and
for every $0 \leq i < n-1$, $b_i = \pi_1(\pi_2(\ldots \pi_2(v_0) \ldots ))$, where $\pi_2$ is repeated $i$ times,
It is easy to see that $c$ and $d$ are mutually inverse.

\begin{prop}[sigma-we-i]
Let $T$ be a theory under $coe_1 + \sigma + Path + wUA + \Sigma$.
Then a map between models of $T$ is a weak equivalence if and only if it has RLP up to $\sim_i$ with respect to $i = i_{(ty,0)}$ and $i = i_{(tm,1)}$.
\end{prop}
\begin{proof}
Since $i_{(ty,0)},i_{(tm,1)} \in \I$, the ``only if'' direction is obvious.
Let us prove the converse.
Let $f : X \to Y$ be a map which has RLP up to $\sim_{i_{(tm,1)}}$ with respect to $i_{(tm,1)}$.
Let $(\Gamma \vdash A) \in X_{(ty,n)}$ and $(f(\Gamma) \vdash a : f(A)) \in Y_{(tm,n)}$.
If $n = 0$, then there exists terms $I \vdash a' : A\!\uparrow$ and $I \vdash p : f(a') = a\!\uparrow$.
Then $\vdash a'[left] : A$ and $\vdash p[left] : f(a'[left]) = a$.
If $n > 0$, then there exists terms $\Sigma(\Gamma) \vdash a' : d^*(A)$ and $f(\Sigma(\Gamma)) \vdash p : f(a') : d^*(a)$.
Then $\Gamma \vdash c^*(a') : A$ and $f(\Gamma) \vdash c^*(p) : f(c^*(a')) = a$.
Thus $f$ has RLP up to $\sim_i$ with respect to every $i \in \I_{tm}$.

Now, suppose that $f$ also has RLP up to $\sim_{i_{(ty,0)}}$ with respect to $i_{(ty,0)}$.
Factor $f$ into a trivial cofibration $g$ followed by a fibration $f'$.
By \cite[lemma~3.4]{f-model-structures}, $f'$ has the same right lifting properties as $f$.
If we can prove that $f'$ is a weak equivalence, then $f$ is a weak equivalence as well by 2-out-of-3 property.
Thus we may assume that $f$ is a fibration.
Then $f$ has RLP with respect to $i_{(ty,0)}$ and $i_{(tm,1)}$.

Let $\Gamma \in X_{(ctx,n)}$ and $(f(\Gamma) \vdash A) \in Y_{(ty,n)}$.
Then there exists a type $A' \in X_{(ty,0)}$ such that $f(A') = \Sigma(\Sigma(\Gamma),d^*(A))$.
There exists a term $A' \vdash a : \Sigma(\Gamma)\!\uparrow$ such that $f(a) = \pi_1(v_0)$.
Now, consider type $(\Sigma(\Gamma) \vdash A'') \in X_{(ty,1)}$ which is defined as $\Sigma(A', a \idtype v_1)$.

Let us prove that there is an equivalence between $f(A'')$ and $d^*(A)$ in context $\Sigma(f(\Gamma))$.
Actually, this is a well-known fact.
For example, it follows from \cite[lemmas 3.11.8 and 3.11.9]{hottbook}.
But since we are working in a restricted context (we do not have $\Pi$ types), we will give a direct proof.

To simplify the notation, let $B = \Sigma(f(\Gamma))$ and $C(t) = d^*(A)[t]$.
A map from $f(A'')$ to $d^*(A)$ is a term $g$ which satisfies
\[ y : B, w : \Sigma(x : \Sigma(z : B, C(z)), \pi_1(x) \idtype y) \vdash g : C(y). \]
Let $t(i) = coe_1(\lambda i.\,C(at(\pi_2(w),i)), \pi_2(\pi_1(w)), i)$.
Then we can define a map $g$ from $f(A'')$ to $d^*(A)$ as $t(right)$.

A map from $d^*(A)$ to $f(A'')$ is a term $g'$ which satisfies
\[ y : B, c : C(y) \vdash g' : \Sigma(x : \Sigma(z : B, C(z)), \pi_1(x) \idtype y). \]
We can define such map $g'$ as $pair(pair(y,c),refl(y))$.
Then $g[g'] = c$ and we can construct a path between $g'[g]$ and $w$ as follows:
\[ path(\lambda i.\,pair(pair(at(\pi_2(w),i), t(i)), path(\lambda j.\,at(\pi_2(w),sq_r(i,j))))) \]

Thus $g$ and $g'$ define an equivalence between $f(A'')$ and $d^*(A)$.
Then $c^*(g)$ and $c^*(g')$ define an equivalence between $f(c^*(A''))$ and $A$.
Thus $\Gamma \vdash c^*(A'')$ is the required lift of $A$.
\end{proof}

If $T$ is under $HPath$, then for every model $X$ of $T$ we can define its homotopy category $Ho(X)$.
To define it, we need to introduce an equivalence relation on the set of terms.
We will say that terms $a,a' \in X_{(tm,n)}$ are equivalent if $ty(a) = ty(a')$ and there exists a term $p$ such that $ctx(a) \vdash p : a = a'$.
Objects of $Ho(X)$ are closed types, that is elements of $X_{(ty,0)}$.
For every $A,B \in X_{(ty,0)}$, morphisms from $A$ to $B$ are equivalence classes of terms $b \in X_{(tm,1)}$ such that $A \vdash b : B\!\uparrow$.
Identity morphism is $v_0$ and composition of $b : A \to B$ and $c : B \to C$ is $subst(c,b)$.
If $x : A \vdash p : b = b'$ and $y : B \vdash q : c = c'$, then $x : A \vdash path(\lambda i.\,at(c,c',q,i)[y \repl at(b,b',p,i)]) : c[y \repl b] = c'[y \repl b']$.
Thus composition is well-defined.

For every morphism $f : X \to Y$ of models of $T$, we define a functor $Ho(f) : Ho(X) \to Ho(Y)$ in the obvious way:
$Ho(f)(A) = f(A)$ for every object $A$, and $Ho(f)(b) = f(b)$ for every morphism $b$.
It is obvious that $Ho$ preserves identity morphisms and compositions.
Thus $Ho$ is a functor $\Mod{T} \to \cat{Cat}$.

If $T$ has $\Sigma$ types, then there is an equivalent characterization of weak equivalences in terms of the homotopy category.
This proposition is similar to \cite[Th\'eor\`eme~3.25]{cis10b}.
Actually, we can probably derive it from results of \cite{cis10b}, but it is easier to give a direct proof.

\begin{prop}[sigma-we-ho]
Let $T$ be a theory under $coe_1 + \sigma + Path + wUA + \Sigma$, and let $f : X \to Y$ be a morphism of models of $T$.
Then a map $f : X \to Y$ between models of $T$ is a weak equivalence if and only if $Ho(f)$ is an equivalence of categories.
\end{prop}
\begin{proof}
Note that $Ho(f)$ is essentially surjective on objects if and only if $f$ has RLP up to $\sim_{i_{(ty,0)}}$ with respect to $\sim_{i_{(ty,0)}}$.

Assume that $f$ is a weak equivalence.
Let $A$ and $B$ be objects of $Ho(X)$, and let $f(A) \vdash b : f(B)\!\uparrow$ be a morphism of $Ho(Y)$.
Then there exists a term $A \vdash b' : B\!\uparrow$ such that $f(b')$ and $b$ are homotopic.
Hence $f(b')$ and $b$ are equals as morphisms of $Ho(Y)$, so $Ho(f)$ is full.
Let $b$ and $b'$ be terms such that $A \vdash b : B\!\uparrow$, $A \vdash b' : B\!\uparrow$ and there exists a term $p$ such that $f(A) \vdash p : f(b) \idtype f(b')$.
Then there exists a term $p'$ such that $A \vdash p' : b \idtype b'$.
Hence $Ho(f)$ is faithful.

Now, assume that $Ho(f)$ is an equivalence of categories.
By \rprop{sigma-we-i}, we just need to prove that $f$ has RLP up to $\sim_{i_{(tm,1)}}$ with respect to $\sim_{i_{(tm,1)}}$.
Factor $f$ intro a trivial cofibration $g$ followed by a fibration $f'$.
Since $Ho(g)$ is an equivalence, $f'$ is also an equivalence by 2-out-of-3 property.
If we can prove that $f'$ is a weak equivalence, then $f$ is a weak equivalence as well by 2-out-of-3 property.
Thus we may assume that $f$ is a fibration.

Let $A \vdash B$ be a type in $X_{(ty,1)}$, and let $b$ be a term in $Y_{(tm,1)}$ such that $f(A) \vdash b : f(B)$.
Since $Ho(f)$ is full, there exists a term $A \vdash b' : \Sigma(A,B)\!\uparrow$ such that $f(b')$ is homotopic to $pair(v_0,b)$.
Since $f$ is a fibration, we can assume that $f(b') = pair(v_0,b)$.
We need to find a term $s$ such that $A \vdash s : \pi_1(b') = v_0$ and $f(s)$ is homotopic to $refl(v_0)$.
If such term exists, then since $f$ is a fibration, there exists a term $s'$ such that $A \vdash s' : \pi_1(b') = v_0$ and $f(s') = refl(v_0)$.
Then we can define $A \vdash b'' : B$ as $coe_0(\lambda i.\,B[at(s',i)], \pi_2(b))$.
This $b''$ is the required lift since $f(b'') = b$.

It is easy to find a term $s$ which satisfies the first condition using the fact that $Ho(f)$ is faithful, but the second condition is more difficult.
Let us show how to construct a term which satisfies both of them.
Since $Ho(f)$ is full, there exists a term $A \vdash s' : \Sigma(A, \pi_1(b') \idtype v_0)\!\uparrow$ such that $f(s')$ is homotopic to $pair(v_0, refl(v_0))$.
Since $f$ is a fibration, we can assume that $f(s') = pair(v_0, refl(v_0))$.
Since $Ho(f)$ is faithful and $f(\pi_1 \circ s') = f(v_0)$, there exists a term $h$ such that $A \vdash h : \pi_1(s') \idtype v_0$.
Let $t(i) = coe_1(\lambda j.\,\pi_1(b')[at(h,j)] \idtype at(h,j), \pi_2(s'), i)$.
If we define $s$ as $t(right)$, then $A \vdash s : \pi_1(b') \idtype v_0$.

Let us prove that $f(s)$ is homotopic to $refl(v_0)$.
If we define $h'$ as
\[ path(\lambda i.\,pair(at(h,i), t(i))), \]
then $A \vdash h' : s' \idtype pair(v_0,s)$.
Hence, we have the following homotopy:
\[ f(A) \vdash f(h') : pair(v_0, refl(v_0)) \idtype pair(v_0, f(s)). \]
If we have $\Gamma \vdash p : \Sigma(A, v_0 \idtype v_0)$, $\Gamma \vdash p' : \Sigma(A, v_0 \idtype v_0)$ and $\Gamma \vdash h : p \idtype p'$,
then it is easy to see that there exist terms $\Gamma \vdash q_1 : \pi_1(p) = \pi_1(p')$ and $\Gamma \vdash q_2 : sym(q_1) * \pi_2(p) * q_1 \idtype \pi_2(p')$.
If we take $p = pair(v_0, refl(v_0))$, $p' = pair(v_0, f(s))$ and $h = f(h')$, then $f(A) \vdash q_2 : sym(q_1) * q_1 \idtype f(s)$,
which implies that $f(s)$ is homotopic to $refl(v_0)$.
\end{proof}

Let $T_\Sigma = coe_1 + \sigma + Path + wUA + \Sigma$.
Every model of $T_\Sigma$ carries the structure of a fibration category, which was proved in \cite{tt-fibr-cat}.
Let us briefly describe this construction.
If $X$ is a model of $T_\Sigma$, then we define category $U(X)$, which has contexts (that is, elements of $X_{(ctx,n)}$) as objects and context morphisms as morphisms.
A map $f : A \to B$ is a fibration if and only if it is isomorphic over $B$ to a map of the form $\pi_1 : \Sigma(B,C) \to B$.
Weak equivalences of $U(X)$ are homotopy equivalences.

The homotopy category $Ho(X)$ is equivalent to the localization of $U(X)$ with respect to homotopy equivalences.
Indeed, first note that since we have $\Sigma$ types, $U(X)$ is equivalent to its full subcategory on contexts of length $\leq 1$.
Let $Ho'(X)$ be the category which has contexts of length $\leq 1$ as objects,
and morphisms of $Ho'(X)$ are equivalence classes of maps of $U(X)$ with respect to the homotopy relation.
Then $Ho'(X)$ is equivalent to the localization of $U(X)$, which can be proved as usual (see, for example, \cite[Corollary~1.2.9]{hovey}).
Then $Ho(X)$ is a full subcategory of $Ho'(X)$.
The only object that $Ho(X)$ lacks is the empty context.
But it contains the interval type which is isomorphic to the empty context in the homotopy category
(both of them are terminal objects in $Ho'(X)$), so $Ho(X)$ is equivalent to $Ho'(X)$.

Thus, \rprop{sigma-we-ho} implies that a map $f : X \to Y$ of models is a weak equivalence if and only if
corresponding map $U(f)$ of fibration categories is a weak equivalence.

\bibliographystyle{amsplain}
\bibliography{ref}

\end{document}